\documentclass{amsart}

\usepackage[a4paper]{geometry}
\usepackage{amsmath}
\usepackage{amsthm}
\usepackage{amsfonts}
\usepackage{amssymb}
\usepackage{mathtools}
\usepackage{eucal}

\usepackage{tikz}
\usetikzlibrary{arrows}

\tikzset{>=stealth}

\tikzset{dbl/.style={double, 
                     double equal sign distance, 
                     -implies, 
                     shorten >=10pt, 
                     shorten <=10pt}}

\tikzset{font={\fontsize{10pt}{12}\selectfont}}   

\DeclareMathOperator{\id}{id}
\DeclareMathOperator*{\colim}{colim}
\DeclareMathOperator{\sing}{Sing}
\DeclareMathOperator{\Diag}{Diag}
\DeclareMathOperator{\ob}{Ob}
\DeclareMathOperator{\mor}{Mor}
\DeclareMathOperator{\Hom}{Hom}
\DeclareMathOperator{\pr}{pr}
\DeclareMathOperator{\Cyl}{Cyl}

\newcommand{\stack}[1]{\mathcal{#1}}
\newcommand{\category}[1]{{\mathsf{#1}}}
\newcommand{\sset}{\category{sSet}}
\newcommand{\presheafset}{\category{pshSet}}
\newcommand{\presheafgpd}{\category{pshGpd}}
\newcommand{\topspace}{\category{Top}}

\newcommand{\sgpd}{\category{sGpd}}
\newcommand{\bsset}{\category{bsSet}}
\newcommand{\topstack}{\category{topStack}}
\newcommand{\set}{\category{Set}}
\newcommand{\gpd}{\category{Gpd}}
\newcommand{\Deltacat}{\mathbf{\Delta}}

\newcommand{\smashedlongrightarrow}{\setbox0=\hbox{$\longrightarrow$}\ht0=2pt\box0}
\newcommand{\risom}{\buildrel\sim\over{\smashedlongrightarrow}}
\newcommand{\rcong}{\buildrel\cong\over{\smashedlongrightarrow}}
\newcommand{\stimes}{\times}
\newcommand{\ttimes}{\tilde{\stimes}}
\newcommand{\Del}{_{\Deltacat}}
\newcommand{\op}{^\text{op}}
\newcommand{\stackvee}{\coprod}

\usepackage{color}
\usepackage{enumerate}

\usepackage{hyperref}                   
\usepackage[capitalise]{cleveref}       

\newtheorem{theorem}{Theorem}[section]
\newtheorem{cor}[theorem]{Corollary}
\newtheorem{lemma}[theorem]{Lemma}
\newtheorem{prop}[theorem]{Proposition}
\theoremstyle{definition}
\newtheorem{defn}[theorem]{Definition}
\theoremstyle{remark}
\newtheorem{rem}[theorem]{Remark}
\numberwithin{equation}{section} \theoremstyle{remark}
\newtheorem{ex}[theorem]{Example}

\begin{document}

\title{Singular chains on topological stacks, I}
\author{Thomas Coyne, Behrang Noohi}
\date{\today}

\begin{abstract}
We extend the functor $\sing$ of singular chains to the category of topological
stacks and establish its main  properties. We  prove that  $\sing$ respects
weak equivalences and takes a morphism of topological stacks that is both a
Serre and a Reedy fibration to a Kan fibration of simplicial sets. When
restricted to the category of topological spaces $\sing$ coincides with
the usual singular functor.
\end{abstract}

\maketitle

\tableofcontents

\section{Introduction}{\label{S:Intro}}

This is the first instalment of a two-part paper 
investigating singular chains on topological stacks. 

Given a topological stack $\stack{X}$ we define the simplicial 
set $\sing(\stack{X})$ of singular chains on $\stack{X}$ and 
establish its main properties. This generalizes the usual functor 
$\sing: \topspace \to \sset$ of singular chains on  topological 
spaces. We address the following questions about $\sing(\stack{X})$: 
functoriality with respect to morphisms of stacks, the homotopy type 
of $\sing(\stack{X})$, and the effect on fibrations of topological stacks.

\medskip\noindent
{\bf Functoriality and the homotopy type of $\sing(\stack{X})$.} There are several 
ways to define the homotopy type of a topological stack (see for instance 
\cite{Behrend, Haefliger, Moerdijk, No12}). In \cite{No12} the notion of 
\emph{classifying space} of a topological stack  is introduced to give a 
better grip on the functoriality of the homotopy type. Nevertheless,  
the functoriality of the classifying space only makes sense in the homotopy 
category of topological spaces, that is, the classifying space is a functor
$\operatorname{CS} : \topstack \to \operatorname{Ho}(\topspace)$.

Our construction of singular chains in this paper enhances this by giving 
us an honest  functor $\sing : \topstack \to \sset$. When restricted to 
the subcategory $\topspace$, this functor coincides with the usual singular 
functor on topological spaces. The functor $\sing : \topstack \to \sset$ in 
fact lifts the classifying space functor $\operatorname{CS}$,
  	\[
\begin{tikzpicture}[scale=1.5]

	\node (a) at (0,1) {$\topstack$};
	\node (b) at (0,0) {$\operatorname{Ho}(\topspace)$};
	\node (c) at (2,1) {$\sset$};
	\node (d) at (2,0) {$\operatorname{Ho}(\sset)$};
	
	\draw[->] (a) to node[left] {$\operatorname{CS}$} (b);
	\draw[->] (a) to node[above] {$\sing$} (c);
	\draw[->] (b) to node[below] {$\sim$} (d);
	\draw[->] (c) to node[right] {} (d);

\end{tikzpicture}
\]
at least for the full subcategory of $\topstack$ consisting of Serre stacks.
This is a consequence of one of our main results (\cref{T:preserves we}).

\begin{theorem}
\label{T:main1}
Let $f : \stack{X}\to\stack{Y}$ be a weak equivalence of Serre stacks. 
Then, $\sing(f) : \sing(\stack{X})\to\sing(\stack{Y})$ is a weak equivalence 
of simplicial sets. In particular, if $X \to \stack{X}$ is a classifying 
space for $\stack{X}$, then the induced map $\sing(X) \to \sing(\stack{X})$
is a weak equivalence.
\end{theorem}

In particular, $\sing(\stack{X})$ has the same homotopy type as the classifying 
space $\operatorname{CS}(\stack{X})$. Somewhat  surprisingly,
the proof of the above theorem is highly nontrivial.

\medskip\noindent
{\bf Effect on fibrations of topological stacks.} It is well known that for a 
Serre fibration $f : X \to Y$  of topological spaces, the induced map 
$\sing(f) : \sing(X) \to \sing(Y)$ is a Kan fibration of simplicial sets. 
The corresponding statement for topological stacks, however, should be formulated 
more carefully, as there are various notions of fibrations between topological 
stacks. For example, the above statement would clearly be false if  we use  the 
notion of Serre fibration for topological stack as in (\cite{No14}, Definition 3.6), 
because this notion is ``intrinsic'' (i.e., is invariant under replacing a stack 
by an equivalent stack -- in particular,  any equivalence of topological stacks 
$f : \stack{X} \to \stack{Y}$, such as the inclusion of a point into a trivial 
groupoid, is automatically a Serre fibration).

It turns out, the correct condition on a morphism $f : \stack{X} \to \stack{Y}$ 
to ensure that $\sing(f) : \sing(\stack{X}) \to \sing(\stack{Y})$ is a Kan fibration 
is that $f$ is a Serre fibration and also a Reedy fibration (\cref{D:reedy_cond}). 
This is another main result of the paper (\cref{T:sing of WSR fib}).

\begin{theorem}
\label{T:main2}
 Let $p : \stack{X}\to \stack{Y}$  be a morphism of Serre topological stacks 
 that is a (weak) Serre fibration and also a Reedy fibration. Then,
 $\sing(p):\sing(\stack{X}) \to \sing(\stack{Y})$ is a (weak) Kan fibration.
\end{theorem}

We point out that the Reedy condition can always be arranged for any morphism 
of stacks: given  $f : \stack{X} \to \stack{Y}$ we can  replace $\stack{X}$ 
by an equivalent stack $\stack{X}'$ such that the corresponding morphism 
$f' : \stack{X}' \to \stack{Y}$ is a Reedy fibration (\cref{P:reedy_replace}). 
Such a replacement would not affect the property of being a Serre fibration.

\medskip
The paper is organized as follows. In  Section \ref{S:TopSt}, we set up 
the terminology and review some generalities about stacks and  topological stacks. 
In Section \ref{S:Tilde}, we introduce the \emph{tilde} construction. 
This is the left Kan extension along the inclusion $\Deltacat \to \topspace$ and 
plays a crucial role in the rest of the paper. In Sections \ref{S:Homotopy}-\ref{S:Lemmas} 
we review some basic facts about homotopy of maps between morphisms of stacks. 
We also recall the relevant background on fibrations of stacks. 
The only new notion in this section is that of a 
\emph{restricted} homotopy (\ref{SS:restricted}) which is related to the 
tilde construction introduced in Section \ref{S:Tilde}.

In Section \ref{S:ModelStr} we look at various model structures on 
the categories of groupoids, presheaves of groupoids and simplicial groupoids, 
and establish some of their properties which are, presumably, well known but which
we have been unable to locate 
 in the literature. The notion of Reedy fibration of stacks (\cref{D:reedy_cond}) 
introduced and studied in this section is central to the paper. 
It is an adaptation of Reedy fibration of simplicial groupoids.

We introduce the functor $\sing : \topstack \to \sset$ in  
Section \ref{S:Sing}. Section \ref{S:Lifting} is the technical heart of the paper 
where we prove a list of lemmas which play key role in the proofs of our main results. 
In Section \ref{S:Kan} we prove the first main result of the paper, namely, that if 
$f : \stack{X} \to \stack{Y}$ is both a Serre and  a Reedy fibration, 
then $\sing(f)$ is a Kan fibration (\cref{T:sing of WSR fib}).

In  Section \ref{S:WE} we use the results of \cref{S:Kan} to prove the 
second main result of the paper, namely, that $\sing$ preserves weak equivalences 
(\cref{T:preserves we}). In particular, this implies that the (singular simplicial set of the) 
classifying space of a topological stack $\stack{X}$ is naturally weakly 
homotopy equivalent to $\sing(\stack{X})$, see \cref{P:classifying_we}.

In the subsequent paper on the subject we study the adjunction between
$\sing$ and geometric realization, as well as the effect of the functor
$\sing$ on the totalization
of cosimplicial stacks. We use these results to study singular chains 
on mapping stacks. 

\subsubsection*{Acknowledgement}
We would like to thank the referees for reading the paper with meticulous
care and suggesting 
numerous corrections and improvements which significantly enhanced the quality of the paper.

\section{Notation and terminology}
\label{Setup}

\subsection{Stacks and Yoneda}
We will use the rather unconventional approach of working with \emph{presheaves of groupoids} 
rather than \emph{categories fibered in groupoids}. We use calligraphic symbols
$\stack{X}$, $\stack{Y}$, etc. for 
presheaves of groupoids.

We often regard a topological space $X$ as a stack via Yoneda embedding. We use the same
notation $X$ for the functor represented by $X$.

\subsection{Strict versus 2-categorical limits}
When we talk about (co)limits  in a 2-category $\category{C}$
we always mean the \emph{strict} ones. 
Otherwise, we call them 2-categorical (co)limits, or 2-(co)limits.

In particular, for (presheaves of) groupoids $\stack{X}$, $\stack{Y}$ and $\stack{Z}$, we denote their 
\[\text{strict fiber product by }  \stack{X}\times_{\stack{Y}}\stack{Z}\] 
and their 
\[\text{2-fiber product by } \stack{X}\ttimes_{\stack{Y}}\stack{Z}.\] 
The notation $\stack{X}\cong\stack{Y}$ means an isomorphism of (presheaves of) groupoids, 
and $\stack{X}\sim\stack{Y}$ means an equivalence of (presheaves of) groupoids.

\subsection{Composition of morphisms in categories}
We use functional notation $g\circ f$ for composition of 1-morphisms $f : X \to Y$ and $g : Y \to Z$, 
and multiplicative notation $\alpha\cdot\beta$ (or simply $\alpha\beta$) for composition of 
2-isomorphisms $\alpha : f \Rightarrow g$ and $\beta : g \Rightarrow h$. 
We use the notation $h \circ \alpha$ for the composition of a 2-isomorphism 
$\alpha : f \Rightarrow g$ between $f,g : X \to Y$ with a morphism $h : Y \to Z$.

\subsection{Categories of interest}
We  usually use the notation $[\category{C},\category{D}]$ for functor categories. 
We will be working with the following categories:
\begin{itemize}
\item[-] $\topspace$, the category of compactly generated Hausdorff spaces;
\item[-] $\gpd$, the category of small groupoids;
\item[-] $\presheafset$, the category of presheaves of sets over $\topspace$;
\item[-] $\presheafgpd$, the category of presheaves of groupoids over $\topspace$;
\item[-] $\sset$, the category of simplicial sets;
\item[-] $\sgpd$, the category of simplicial groupoids;
\item[-] $\bsset$, the category of bisimplicial sets.
\end{itemize}

Note that $\gpd$, $\presheafgpd$ and  $\sgpd$ carry a 2-category structure; we will
use the same notation for the corresponding 2-categories.

\subsection{Simplicial sets}
The category of finite ordinal numbers with order preserving maps between them 
is denoted by $\Deltacat$. The simplicial $n$-simplex is denoted by $\Delta^n:=\Hom_{\Deltacat}(-,[n])$. 
The topological $n$-simplex is denoted
by 	
\[
		|\Delta^n| = \{(x_0,x_1,\dots,x_n)\in \mathbb{R}^{n+1} :  \sum_{i=0}^nx_i=1,x_i\geq0 \}
\]
We denote the cosimplicial object $n \mapsto |\Delta^n|$ in $\topspace$ by $|\Delta^{\bullet}|$. 
The $k^{\text{th}}$ horn in $\Delta^n$, namely,  the sub-simplicial set of $\Delta^n$ 
generated by the  $i^{\text{th}}$ faces of the unique non-degenerate $n$-cell in 
$\Delta^n$, $i\in\{0,1,\cdots,\hat{k},\cdots,n\}$, is denoted  by $\Lambda^n_k$. 
When talking about homotopies between maps we often use the
notation $[0,1]$ instead of $|\Delta^1|$.

The bisimplex $\Delta^{m,n}$ is the bisimplicial set
$\Delta^{m,n}:\Deltacat^{\text{op}} \times \Deltacat^{\text{op}} \to \set$
represented by $([m],[n]) \in \Deltacat \times \Deltacat$. That is,
$\Delta^{m,n}:=\Hom_{\Deltacat \times \Deltacat}(-,([m],[n]))= \Delta^{m}\boxtimes\Delta^{n}$
(see \cref{SS:diagonal}).

For a simplicial set $X \in \sset$, we use the notation
$\tilde{X} \in \presheafgpd$ for the left Kan extension
of $X$ along $\Deltacat \to \topspace$ (more details can be found in \cref{S:Tilde}).

\section{Topological stacks}
\label{S:TopSt}

Throughout the paper, we will work over the base Grothendieck site  $\topspace$    
of  compactly generated Hausdorff topological spaces (with the open-cover topology). 
We will use the rather unconventional approach of working with \emph{presheaves of 
groupoids} rather than \emph{categories fibered in groupoids} (there is a natural
strictification functor from the latter to the former). The equivalence of this approach 
with Grothendieck's approach via fibered categories has been worked out in 
\cite{Hollander08} (also see \cref{SS:2catstacks} below).

\subsection{Presheaves of groupoids}
We denote the 2-category of presheaves of groupoids over $\topspace$ by 
$\presheafgpd=[\topspace^\text{op},\category{Gpd}]$. 
We shall denote the objects of this category with calligraphic letters, 
i.e., $\stack{X} \in \presheafgpd$. For $T \in  \topspace$ 
we call $\stack{X}(T)$ the \emph{groupoid of $T$-points} of $\stack{X}$.

By an \emph{equivalence} of presheaves of groupoids we mean a morphism 
$f : \stack{X}  \to \stack{Y}$ such that for every $T \in  \topspace$, 
the induced map $f(T) : \stack{X}(T)  \to \stack{Y}(T)$ on the $T$-points 
is an equivalence of groupoids. Two presheaves of groupoids are \emph{equivalent} 
if there exists a zigzag of equivalences between them.

\subsection{Yoneda}
Let $\presheafset = [\topspace\op,\set]$ be the category of presheaves of sets over 
the category $\topspace$ of topological spaces. Regarding a set  as a
a groupoid in which 
the only morphisms are the identity morphisms,
we identify $\presheafset$
with a full subcategory $\presheafgpd$ of 
the category of presheaves of groupoids over 
the category $\topspace$.

The Yoneda functor  $\topspace \to \presheafset$ (or $\topspace \to \presheafgpd$)
sends a topological space $X$ to the functor $\Hom_{\topspace}(-,X)$
represented by $X$. This identifies $\topspace$ with a full subcategory $\presheafset$ 
(or $\presheafgpd$). More precisely, we have the following.

\begin{lemma}{\label{L:Yoneda}}
  Let $X$ be a topological space and $\stack{Y}$ a presheaf of groupoids. Then, there is 
  a natural 
  isomorphism of groupoids $\Hom_{\presheafgpd}(X,\stack{Y})\cong \stack{Y}(X)$.
\end{lemma}

As  in the above lemma, we often abuse notation and use 
the same notation both for $X\in \topspace$ and for the image of $X$ in 
$\presheafset$ (or $\presheafgpd$) under the Yoneda functor.

The Yoneda embedding preserves fiber products (in fact, all limits), but it 
seldom preserves colimits. If a presheaf $\stack{X}$ is equivalent to $\Hom_{\topspace}(-,X)$, 
for some $X \in \topspace$, we often abuse terminology and say that $\stack{X}$ is 
a topological space.

\subsection{Stacks}{\label{SS:2catstacks}}
Following (\cite{Hollander08}, Definition 1.3) we define a 
\emph{stack} over $\topspace$  to be a presheaf of groupoids 
$\stack{X}\in \presheafgpd$ that satisfies the descent condition
\[
\begin{tikzpicture}
	
	\node (a) at (0,0) {$\stack{X}(T)$};
	\node (b) at (3,0) {$\text{holim} \Big ( \prod \stack{X}(U_i)$};
	\node (c) at (6,0) {$\prod \stack{X}(U_{ij})$};
	\node (d) at (9,0) {$\prod \stack{X}(U_{ijk})\Big )$};
	
	\draw[->] (a) to  node[above] {$\sim$} (b);
	
	\draw[->,transform canvas={yshift=0.5ex}] (b) to (c);
	\draw[->,transform canvas={yshift=-0.5ex}] (b) to (c);
	
	\draw[->,transform canvas={yshift=0.8ex}] (c) to (d);
	\draw[->] (c) to (d);
	\draw[->,transform canvas={yshift=-0.8ex}] (c) to (d);
	
\end{tikzpicture}
\]
for every $T \in \topspace$ and every open cover $\{U_i\}$ of $T$. Morphisms and 
2-isomorphisms of stacks are the ones of the underlying presheaves of groupoids. 
That is, stacks form a full sub-2-category of $\presheafgpd$.

As shown in \cite{Hollander08}, 
the presheaf approach to stacks  is equivalent
to the approach via categories fibered in groupoids. Let us elaborate on this.
The projective model structure on (the 1-category underlying) 
$\presheafgpd$ 
 is Quillen equivalent to the projective model structure (in the sense of
 \cite{Hollander08}, Theorem 4.2) on 
(the 1-category of) categories fibered in groupoids
over $\topspace$; see \cite{Hollander08}, Corollary 4.3.
The underlying Quillen adjunction is defined as follows: to any category
$\stack{C}$ fibered in groupoids over $\topspace$ we associate the presheaf of groupoids 
  \[T \in \topspace, \ \  T \mapsto \Hom_{\operatorname{FibCat}}(T,\stack{C})  \ \in \gpd.\]
The left adjoint to this functor is given by the Grothendieck construction. 

 This Quillen equivalence gives rise to a Quillen 
equivalence between the localizations of both model categories 
with respect to hypercovers;
see \cite{Hollander08}, Corollary 4.5. The fibrant
objects in either of these localized model categories are called stacks.

\subsection{Topological stacks}
By a \emph{topological stack} we mean a  stack over $\topspace$ 
which is equivalent to the quotient stack of a topological groupoid 
$\mathbb{X}=[R\rightrightarrows X]$, with $R$ and $X$ topological spaces.  
A topological stack is \emph{Serre} if it has a groupoid presentation such that 
$s : R \to X$ is  locally (on source and target) a Serre fibration. 
That is, for every $y \in R$, $s$ is a Serre fibration from a neighborhood of $y$ to a
neighborhood of $s(y)$.

Morphisms and 2-isomorphisms of topological stacks are the ones of the underlying 
presheaves of groupoids, so topological stacks, as well as Serre topological stacks, 
form a full sub-2-category of $\presheafgpd$. We denote the 2-category of topological 
stacks by $\topstack$.

\subsection{Strict and 2-categorical fiber products}
\label{SS:strict}
Consider the following diagram in the 2-category $\gpd$ of groupoids:
\[
\begin{tikzpicture}[scale=1]

	\node (b) at (0,0) {$H$};
	\node (c) at (2,2) {$K$};
	\node (d) at (2,0) {$G$};
	
	\draw[->] (b) to node[below] {q} (d);
	\draw[->] (c) to node[right] {p} (d);

\end{tikzpicture}
\]
Recall that the \emph{2-fiber product} (or \emph{2-categorical fiber product}) 
  \[H\ttimes_G K\]
has objects triples $(x,y,\varphi)$, where $x$ is an  object in $H$, $y$ is an object 
in  $K$ and $\varphi : q(x) \to p(y)$ is a morphism in $G$. A morphism from
$(x,y,\varphi)$ to $(x',y',\varphi')$ is a pair of morphisms $\alpha :x \to x'$ and 
$\beta : y \to y'$, in $H$ and $K$ respectively, such that 
$\varphi'\circ q(\alpha) =p(\beta)\circ \varphi$.

There is a fully faithful functor 
       \[H\times_G K \to H\ttimes_G K\]
from the strict fiber product to the 2-fiber product, sending a  pair 
$(x,y) \in  H\times_G K$ to the triple $(x,y,\id)$. The image consists of those
triples  $(x,y,\varphi)$ with $\varphi=\id$.
This map is sometimes an equivalence
(Lemma \ref{L:strictvs2-categorical}) but not always.

The strict and 2-categorical product are defined objectwise for presheaves of groupoids,
namely
    \[(\stack{X}\times_{\stack{Z}}\stack{Y})(T)=\stack{X}(T)\times_{\stack{Z}(T)}\stack{Y}(T),
   \ \  \forall T \in \topspace\]
 and  
\[(\stack{X}\ttimes_{\stack{Z}}\stack{Y})(T):=\stack{X}(T)\ttimes_{\stack{Z}(T)}\stack{Y}(T),
   \ \  \forall T \in \topspace.\]

\begin{lemma} 
The 2-categories of stacks,  topological stacks and Serre topological stacks are all closed under 2-fiber products (in fact, all finite 2-limits), and these are computed as presheaves of groupoids.
\end{lemma}

\begin{proof}
In the case of stacks this is well known (homotopy limit commutes with 
2-fiber product). For the other two cases see (\cite{No05}, page 30) for the 
construction of a groupoid presentation for $\stack{X}\times_{\stack{Z}}\stack{Y}$ 
out of those for  $\stack{X}$, $\stack{Y}$ and $\stack{Z}$. 
\end{proof}

\begin{rem}
The reason for using the nonstandard notation $\ttimes$ is that in this paper we will 
mostly be using \emph{strict} fiber products of (presheaves of) groupoids  and we need to
distinguish between the two notions.
\end{rem}

%

\subsection{Classifying spaces for  topological stacks}

The following theorem has been proven in (\cite{No14}, Corollary 3.17).
\begin{theorem}
\label{T:classifying}
Let $\stack{X}$ be a topological stack. Then, there exists 	an atlas 
$\varphi :X \to \stack{X}$ that is a trivial weak  Serre fibration. This means that, 
for any map from a topological space $T$, the fiber product
\[
\begin{tikzpicture}[scale=1.5]

	\node (a) at (0,1) {$X \ttimes_{\stack{X}} T$};
	\node (b) at (0,0) {$X$};
	\node (c) at (2,1) {$T$};
	\node (d) at (2,0) {$\stack{X}$};
	
	\draw[->] (a) to node[left] {} (b);
	\draw[->] (a) to node[above] {} (c);
	\draw[->] (b) to node[below] {$\varphi$} (d);
	\draw[->] (c) to node[right] {} (d);

\end{tikzpicture}
\]
has the property that $X \ttimes_{\stack{X}} T \to T$ is a trivial  weak Serre 
fibration of topological spaces (in particular, a weak homotopy equivalence).
\end{theorem}

See \cref{D:WTS_fib} for the general definition of trivial weak  Serre fibration, 
bearing in mind that the definition simplifies considerably in the case of 
topological spaces.

We call a map $\varphi :X \to \stack{X}$ as in \cref{T:classifying} a 
\emph{classifying atlas} for $\stack{X}$. Note that in the definition of 
classifying atlas given in \cite{No12} we only require $\varphi :X \to \stack{X}$ to be a
universal weak equivalence. The definition we are using here is stronger.

The $n^{\text{th}}$ \emph{homotopy group} (set if $n=0$) of a pointed topological 
stack $(\stack{X},x)$ is defined (\cite{No14}, Section 5) to be the group 
$\pi_n(\stack{X},x) = [(S^n,s_0),(\stack{X},x)]$ of homotopy  classes of pointed 
maps. Equivalently, it can be defined to be the homotopy group $\pi_n(X,x')$ of a 
classifying atlas $X$ for $\stack{X}$ at some lift $x'$ of $x$ to $X$. This 
definition is independent of the choice of $X$ and $x'$ (up to a natural isomorphism).

A morphism of topological stacks $f:\stack{X} \to \stack{Y}$ is called a 
\emph{weak equivalence} if it induces isomorphisms 
$f_*:\pi_n(\stack{X},x) \to \pi_n(\stack{Y},y)$  for all choices of basepoint 
and all $n\geq 0$.

\section{The tilde construction}
\label{S:Tilde}
Consider the inclusion $\Deltacat \to \topspace$, $[n] \mapsto |\Delta^n|$. 
Left Kan extension along this inclusion gives rise to a functor
\begin{align*}
  \sset & \to \presheafset \ \ (\hookrightarrow \presheafgpd) \\
   A & \mapsto \tilde{A}
\end{align*}
which is uniquely determined by the property that it preserves colimits
and sends $\Delta^n$ to $|\Delta^n|$ (rather, the presheaf represented by it).
It is left adjoint to the restriction functor
\begin{align*}
	-\Del : \presheafset &\to \sset \ \ (\hookrightarrow \sgpd) \\
	X &\mapsto X \Del =
	\Hom_{\presheafset}(|\Delta^{\bullet}|,X).
\end{align*}

More explicitly, $\tilde{A}$ is constructed exactly like the colimit   
construction of the geometric realization of $A$, except that instead of 
using the topological simplices $|\Delta^n|$  as building blocks
we use the presheaves in $\presheafset$ represented by them.

We have a natural map
    \begin{equation}\label{Eq:psi} \psi_A : \tilde{A} \to |A|.    
    \end{equation}
This is adjoint to the map $A \to \sing(|A|)=|A|\Del$, the unit of the adjunction 
$| - |:\topspace\rightleftharpoons\sset: \sing$. Note that the Yoneda embedding 
$\topspace \to \presheafset$   (or $\presheafgpd$) does not necessarily preserve 
colimits, so $\psi_A$ is often not an isomorphism (but it is
when $A=\Delta^n$).

%

\begin{rem}
 The standard notation in the literature for the 
 restriction functor along $\Deltacat \to \topspace$ (rather, along
 $i : \Deltacat^{\text{op}} \to \topspace^{\text{op}}$)  and its left adjoint, 
 the left Kan extension,  are $i^*$ and $i_!$, respectively. Our 
 choice of the alternative notation $( )_{\Delta}$ and $()^{\sim}$
 is only to reduce the burden of notation and enhance readability of the long formulas
 we will encounter.
\end{rem}

\begin{ex}
\label{E:Lambda}
	Write $\Lambda^n_k$ as the coequalizer of
	\begin{align*}
		\coprod_{0\leq i<j\leq n} \Delta^{n-2}
		\rightrightarrows \coprod_{i\in\{0,1,\dots,n\},i\neq k} \Delta^{n-1} \to \Lambda^n_k
	\end{align*}
	Then, we can write $\tilde{\Lambda}^n_k$ as the coequalizer
	\begin{align*}
		\coprod_{0\leq i<j\leq n} |\Delta^{n-2}| \rightrightarrows
		\coprod_{i\in\{0,1,\dots,n\},i\neq k} |\Delta^{n-1}|
	\end{align*}
	in $\presheafset$.
	The map $\psi_{{\Lambda}^n_k} :\tilde{\Lambda}^n_k \to |\Lambda^n_k|$ is almost never
	an isomorphism.
\end{ex}
	
We can extend the restriction functor $-\Del$ defined above
to $\presheafgpd$:
\begin{align*}
	-\Del : \presheafgpd &\to \sgpd,\\
	\stack{X} &\mapsto \stack{X} \Del = \Hom_{\presheafgpd}(|\Delta^{\bullet}|,\stack{X}).
\end{align*}
We have the following lemma.

\begin{lemma}
\label{L:restriction isom}
Let $A$ be a simplicial set and  $\stack{X}$ a presheaf of groupoids. Then, 	
we have an isomorphism (and not just an equivalence) of groupoids
	\begin{align*}
		\Hom_{\presheafgpd}(\tilde{A},\stack{X}) & \stackrel{\cong}{\longrightarrow}
		\Hom_{\sgpd}(A,\stack{X}\Del),  \\
		f & \mapsto f\Del\circ \iota_A.
	\end{align*}
Here, $\iota_A: A \to \tilde{A}\Del$ is the unit of adjunction.
In particular, we have the following natural isomorphisms
\[
\begin{tikzpicture}[scale=1]

	\node (a) at (0,0) {$\Hom_{\presheafgpd}(|\Delta^n|,\stack{X})$};
	\node (b) at (6,0) {$\Hom_{\sgpd}(\Delta^n,\stack{X}_{\Del})$};
	\node (c) at (3,-2) {$\stack{X}(|\Delta^n|)$};
	
	\draw[->] (a) to  node[above] {$\cong$} (b);
	\draw[->] (a) to  node[below] {$\cong$} (c);
	\draw[->] (b) to  node[below] {$\cong$} (c);

\end{tikzpicture}
\]	
	
\end{lemma}

\begin{proof}
In the case where $\stack{X}$ is a presheaf of sets, i.e., $\stack{X} \in \presheafset$, 
this is just the  left adjointness of the left Kan extension. For the general case 
view $\stack{X}$ as a groupoid object in 
$\presheafset$ and apply the above isomorphisms to $\ob(\stack{X})$ and $\mor(\stack{X}) \in \presheafset$.
\end{proof}

\section{Yoneda and colimits}
As we pointed out in the previous section, unless $A$ is representable, the natural map 
$\psi_A : \tilde{A}\to |A|$  is not in general an isomorphism of presheaves of sets. 
This is due to the fact that the Yoneda functor 
$\topspace\to \presheafset$ (or $\topspace\to \presheafgpd$) does not preserve colimits.

\medskip

In certain situations, however, we have the following partial result.

\begin{lemma}
\label{L:fakification}
Let $\stack{X}$ be a Serre topological stack. Let $A\hookrightarrow B$ and 
$A\hookrightarrow C$ be  	closed embeddings of topological spaces. Assume both maps
are locally trivial Serre cofibrations. Then, the map
	\[
		\Hom_{\presheafgpd}(B \stackvee_{A} C,\stack{X})
		\to \Hom_{\presheafgpd}(B{\coprod_{A}}' C,\stack{X})
	\]
induced by the natural map $B\coprod'_{A} C \to B\stackvee_{A} C$ is an 
equivalence of groupoids.  Here, $\stackvee$ stands for colimit in 
$\topspace$ and $\coprod'$ stands for colimit in $\presheafset$
(which is the same as colimit in $\presheafgpd$).
\end{lemma}

\begin{proof}
This is an easy consequence of (\cite{BGNX}, Proposition 1.3). Note that  
(\cite{BGNX}, Proposition 1.3) is proved for Hurewicz stacks. The 
proof for the case of  Serre topological stacks is entirely similar;
see (\cite{No05}, Proposition 16.1 and Theorem 16.2) for more details.

To prove the lemma, note that the groupoid 
\[\Hom_{\presheafgpd}(B{\coprod_{A}}' C,\stack{X})\cong
\Hom_{\presheafgpd}(B,\stack{X})\times_{\Hom_{\presheafgpd}(A,\stack{X})}
   \Hom_{\presheafgpd}(C,\stack{X})\]
can be identified with the full subgroupoid of the groupoid
  \[\Hom_{\presheafgpd}(B,\stack{X})\ttimes_{\Hom_{\presheafgpd}(A,\stack{X})}
   \Hom_{\presheafgpd}(C,\stack{X}) \]
consisting of those triples $(f,g,\varphi)$,
  \[ f : B\to \stack{X}, \  g : G\to \stack{X}, \ \varphi : f|_A \Rightarrow g|_A, \]
for which $f|_A=g|_A$ and $\varphi=\id$. 
The composition  
  \begin{multline*}   
   \Hom_{\presheafgpd}(B \stackvee_{A} C,\stack{X})  
      \to \Hom_{\presheafgpd}(B{\coprod_{A}}' C,\stack{X}) \\
     \hookrightarrow  \Hom_{\presheafgpd}(B,\stack{X})\ttimes_{\Hom_{\presheafgpd}(A,\stack{X})}
   \Hom_{\presheafgpd}(C,\stack{X}) 
   \end{multline*}  
is an equivalence of groupoids by (the Serre version) of  (\cite{BGNX}, Proposition 1.3). 
Since the second functor is fully faithful, it follows that both functors are
equivalences of groupoids.  
\end{proof}

\begin{rem}{\label{R:strictvsweak}}
In the course of the proof of the above lemma we have also shown that the natural map
  \begin{multline*} 
  \Hom_{\presheafgpd}(B,\stack{X})\times_{\Hom_{\presheafgpd}(A,\stack{X})}
   \Hom_{\presheafgpd}(C,\stack{X})  \\
   \hookrightarrow
   \Hom_{\presheafgpd}(B,\stack{X})\ttimes_{\Hom_{\presheafgpd}(A,\stack{X})}
   \Hom_{\presheafgpd}(C,\stack{X}) 
    \end{multline*}  
 is an equivalence of groupoids. In other words, the strict and the  2-fiber product are
 equivalent.  
\end{rem}

\begin{defn}{\label{D:glue}}
 We say that a simplicial set $A$ has the {\em gluing property} with respect to a presheaf
 of groupoids $\stack{X}$ if the map
    \begin{align*}
		\Sigma_{A,\stack{X}} : \Hom_{\presheafgpd}(|A|,\stack{X})
		& \to \Hom_{\presheafgpd}(\tilde{A},\stack{X})\\
		f &\mapsto  f\circ \psi_A
	\end{align*}
is an equivalence of groupoids. 
\end{defn}

\begin{lemma}
\label{L:fakification0}
The simplicial $n$-simplex $\Delta^n$ has the gluing property with respect to any 
presheaf of groupoids $\stack{X}$.
\end{lemma}

\begin{proof}
This follows from the fact that $\psi_A : |A| \to \tilde{A}$ is an isomorphism when
$A=\Delta^n$. In fact, in this case
the maps $\Sigma_{\Delta^n,\stack{X}}$ are isomorphisms of groupoids.
\end{proof}

\begin{lemma}
\label{L:fakification1}
Let $A\hookrightarrow B$ and 
$A\hookrightarrow C$ be monomorphisms of simplicial sets. If $A$, $B$ and $C$ have the gluing 
property with respect to a Serre topological stack $\stack{X}$, then so does
$B \stackvee_{A} C$.
\end{lemma}

\begin{proof} 
We have  
 \begin{align*}
 \Hom_{\presheafgpd}(|B \stackvee_{A} C|,\stack{X})
      &  \rcong  \Hom_{\presheafgpd}(|B| \stackvee_{|A|} |C|,\stack{X}) \\
   \tag{\cref{L:fakification}} 
      & \risom   \Hom_{\presheafgpd}(|B| {\coprod_{|A|}}' |C|,\stack{X}) \\     
    \tag{Definition of colimit}  & \rcong \Hom_{\presheafgpd}(|B|,\stack{X})\times_{\Hom_{\presheafgpd}(|A|,\stack{X})}
   \Hom_{\presheafgpd}(|C|,\stack{X})  \\
     \tag{\cref{R:strictvsweak}}  
      & \risom \Hom_{\presheafgpd}(|B|,\stack{X})\ttimes_{\Hom_{\presheafgpd}(|A|,\stack{X})}
   \Hom_{\presheafgpd}(|C|,\stack{X})  \\
   \tag{Assumption}     
      &    \risom \Hom_{\presheafgpd}(\tilde{B},\stack{X})
        \ttimes_{\Hom_{\presheafgpd}(\tilde{A},\stack{X})}
   \Hom_{\presheafgpd}(\tilde{C},\stack{X})  \\
\end{align*}
  Notice that the above equivalence is equal to the following composition:
   \begin{multline*}   
   \Hom_{\presheafgpd}(|B \stackvee_{A}C|,\stack{X})  
      \xrightarrow{\Sigma_{B \stackvee_{A}C,\stack{X}}}
       \Hom_{\presheafgpd}(\widetilde{B \stackvee_{A}C},\stack{X}) \cong \\ 
      \Hom_{\presheafgpd}(\tilde{B} \stackvee_{\tilde{A}}\tilde{C},\stack{X}) 
        \cong\Hom_{\presheafgpd}(\tilde{B},\stack{X})
        \times_{\Hom_{\presheafgpd}(\tilde{A},\stack{X})}
   \Hom_{\presheafgpd}(\tilde{C},\stack{X}) \\
     \hookrightarrow  \Hom_{\presheafgpd}(\tilde{B},\stack{X})
        \ttimes_{\Hom_{\presheafgpd}(\tilde{A},\stack{X})}
   \Hom_{\presheafgpd}(\tilde{C},\stack{X}). 
   \end{multline*}  
Since the last functor is fully faithful and the composition is shown
above to be an equivalence, it follows 
 that $\Sigma_{B \stackvee_{A}C,\stack{X}}$
is also an equivalence.
\end{proof}

Recall that a simplicial set $X$ is called \emph{non-singular}
(\cite{JRW}, Definition 1.2.2) if for every 
non-degenerate $n$-simplex
$x$, the corresponding map $\bar{x} : \Delta^n \to X$ is  a monomorphism. Examples 
we will encounter include
A=$\partial\Delta^n$, $\Lambda^n_k$ and $\Lambda^n_k\times\Delta^1$. Non-singular 
simplicial sets are closed under taking sub-objects and products.

\begin{cor}
\label{C:fakification2}
Let $D$ be a finite non-singular  simplicial set.
Then, $D$ has the gluing property with respect to every Serre topological stack
$\stack{X}$. That is, for every Serre topological stack
$\stack{X}$,
the map $\psi_D : \tilde{D}\to |D|$  induces an equivalence of groupoids
      \begin{align*}
		\Sigma_{D,\stack{X}} :\Hom_{\presheafgpd}(|D|,\stack{X})
		& \risom \Hom_{\presheafgpd}(\tilde{D},\stack{X})\\
		f &\mapsto  f\circ \psi_D.
	\end{align*}
\end{cor}

\begin{proof}
Proof proceeds by induction on the total number of non-degenerate simplices of $D$. 
Choose a maximal non-degenerate
simplex $x$, and write $B \subset D$ for the sub simplicial set of $D$ 
generated by the rest of the
non-degenerate simplices. Set $A:=B\cap \bar{x}(\Delta^n)$, where
$\bar{x} :  \Delta^n \to D$ is the map corresponding to $x$; 
note that this map is a monomorphism by assumption. Also, note that  the sets 
of non-degenerate simplices of  $A$ and $B$ are both properly contained in the set of 
non-degenerate simplices of $D$, so they have a smaller size. 
By the induction hypothesis, the claim is true
for  $A$ and  $B$, and by \cref{L:fakification0} it is also true for $\Delta^n$.
Therefore,  by \cref{L:fakification1}, the claim is true for $D=B\stackvee_{A}\Delta^n$.
\end{proof}

As we pointed out above, in the case 
$D=\Delta^n$ the above equivalence is indeed an isomorphism of groupoids.

\section{Homotopy between morphisms of presheaves of groupoids}
\label{S:Homotopy}

We review the notion of homotopy between morphisms of stacks from \cite{No14},
and introduce a variant called restricted homotopy.

\subsection{Fiberwise homotopy}

\begin{defn}
\label{D:stack_homotopy}
	Let $f,g:\stack{A} \to \stack{X}$ and  $p:\stack{X} \to \stack{Y}$
	be morphisms of presheaves of
	groupoids, and 	$\varphi :p \circ f \Rightarrow p \circ g$ a 2-isomorphism:
	\[
	\begin{tikzpicture}[scale=1.75]

	\node (a) at (0,0) {$\stack{A}$};
	\node (b) at (1,1) {$\stack{X}$};
	\node (c) at (1,0) {$\stack{Y}$};
	
	\draw[->] (a.north) to [bend left = 20] node[above] {\scriptsize $f$} (b.west);
	\draw[->] (a) to [bend right = 0] node[above] {\scriptsize $g$} (b);
	\draw[->] (a.east) to [bend left = 15] node[above,inner sep=1pt] (pf)
	{\scriptsize $p \circ f$} (c.west);
	\draw[->] (a.south east) to [bend right = 15] node[below,inner sep=1pt] (pg)
	{\scriptsize $p \circ g$} (c.south west);
	\draw[->] (b) to node[right] {\scriptsize $p$} (c);
	
	\draw[double,
              double equal sign distance,
              -implies,
              shorten >= 2pt,
              shorten <= 2pt
              ]
              (pf) to node[left] {$\varphi$} (pg);
	
\end{tikzpicture}
\]
A \emph{fiberwise homotopy from $f$ to $g$ relative to}
$\varphi$ is a quadruple  $(H,\epsilon_0,\epsilon_1,\psi)$  where
	\begin{itemize}
		\item $H:\stack{A} \times [0,1] \to \stack{X}$ is a morphism of
		presheaves of groupoids;
		\item $\epsilon_0:f \Rightarrow H_0$ and
           $\epsilon_1:H_1 \Rightarrow g$
		   are 2-isomorphisms;
		\item $\psi:p\circ f \circ \pr_1 \Rightarrow p \circ H$ is a
           2-isomorphism,
		\[
	\begin{tikzpicture}[scale=1.5]

	\node (a) at (0,1) {$\stack{A} \times [0,1]$};
	\node (b) at (0,0) {$\stack{A}$};
	\node (c) at (2,1) {$\stack{X}$};
	\node (d) at (2,0) {$\stack{Y}$};
	
	\draw[->] (a) to node[left] {$\pr_1$} (b);
	\draw[->] (a) to node[above] {$H$} (c);
	\draw[->] (b) to node[below] {$p \circ f$} (d);
	\draw[->] (c) to node[right] {$p$} (d);
	
	\draw[dbl] (0,0.5) to [bend right=35] node[below right] {$\psi$} (1,1);
	
	\end{tikzpicture}
		\]
		such that $\psi_0  = p \circ \epsilon_0$ and
		  $\psi_1 \cdot (p \circ \epsilon_1) = \varphi$.
	\end{itemize}
(Notation: $H_i:=H|_{\stack{A} \times \{i\} }$, 
$\psi_i:= \psi|_{\stack{A} \times \{i\} }$, for $i=0,1$.) 
In the case where $\varphi$ and $\psi$ are both identity 2-isomorphisms 
(so $p \circ f=p \circ g$ and $p\circ f \circ \pr_1 = p \circ H$) 
we say that $H$ is a \emph{homotopy relative to $\stack{Y}$}.

A fiberwise homotopy as above is called \emph{strict} if $\epsilon_0$ 
and $\epsilon_1$ are the identity 2-isomorphisms. 

A \emph{ghost fiberwise homotopy from $f$ to $g$ relative to} $\varphi$ 
is a 2-isomorphism $\xi : f \Rightarrow g$ such that $\varphi = p \circ \xi$. 
\end{defn}

Ghost homotopies typically arise from those quadruples 
$(H,\epsilon_0,\epsilon_1,\psi)$ for which $H$ and $\psi$ remain constant along $[0,1]$, 
that is, they  factor through $\pr_1$. In this case, $\xi:=\epsilon_0\cdot\epsilon_1$ 
is a ghost fiberwise homotopy from $f$ to $g$ relative to $\varphi$. Conversely, from 
a ghost homotopy $\varphi$ we can construct  quadruples 
$(g\circ\pr_1,\xi,\id,\varphi\circ\pr_1)$ and $(f\circ\pr_1,\id, \xi,\id\circ\pr_1)$.

\begin{rem}{\label{R:fiberwisehomotopy}}
 There is some flexibility in choosing $H$. More precisely, if 
 $H':\stack{A} \times [0,1] \to \stack{X}$ is 2-isomorphic to $H$ via 
 $\alpha : H \Rightarrow H'$, then $(H',\epsilon'_0,\epsilon'_1,\psi')$
 is also a fiberwise homotopy from $f$ to $g$ relative to
$\varphi$, where $\epsilon'_0=\epsilon_0\cdot\alpha_0$, 
$\epsilon'_1=\alpha_1^{-1}\cdot\epsilon_1$ and $\psi'=\psi\cdot(p\circ\alpha)$.
\end{rem}

\subsection{Restricted fiberwise homotopy}
\label{SS:restricted}
The notion of restricted homotopy we introduce below only applies to 
morphisms of the form $\tilde{A} \to \stack{X}$, where $A$ is a 
simplicial set and $\stack{X}$ is a presheaf of groupoids.

\begin{defn}
\label{D:fake homotopy}
Let $A$ be a simplicial set. Let  $f,g:\tilde{A} \to \stack{X}$ and 
$p:\stack{X} \to \stack{Y}$ be morphisms of presheaves of groupoids, 
and $\varphi :p \circ f \Rightarrow p \circ g$ a 2-isomorphism:
	\[
\begin{tikzpicture}[scale=1.75]

	\node (a) at (0,0) {$\tilde{A}$};
	\node (b) at (1,1) {$\stack{X}$};
	\node (c) at (1,0) {$\stack{Y}$};
	
	\draw[->] (a.north) to [bend left = 20] node[above] {\scriptsize $f$} (b.west);
	\draw[->] (a) to [bend right = 0] node[above] {\scriptsize $g$} (b);
	\draw[->] (a.east) to [bend left = 15] node[above,inner sep=1pt] (pf)
	{\scriptsize $p \circ f$} (c.west);
	\draw[->] (a.south east) to [bend right = 15] node[below,inner sep=1pt] (pg)
	{\scriptsize $p \circ g$} (c.south west);
	\draw[->] (b) to node[right] {\scriptsize $p$} (c);
	
	\draw[double,
              double equal sign distance,
              -implies,
              shorten >= 2pt,
              shorten <= 2pt
              ]
              (pf) to node[left] {$\varphi$} (pg);
	
\end{tikzpicture}
\]
	A \emph{restricted fiberwise homotopy from $f$ to $g$ relative to}
	$\varphi$ is a quadruple  $(H,\epsilon_0,\epsilon_1,\psi)$  where
	\begin{itemize}
		\item $H:\widetilde{A \times \Delta^1} \to \stack{X}$ is a morphism of
		presheaves of groupoids;
		\item $\epsilon_0:f \Rightarrow H_0$ and
		   $\epsilon_1:H_1 \Rightarrow g$
		   are 2-isomorphisms;
		\item $\psi:p\circ f \circ \widetilde{\pr}_1 \Rightarrow p \circ H$ 
          is a 2-isomorphism,
		\[
	\begin{tikzpicture}[scale=1.5]

	\node (a) at (0,1) {$\widetilde{A \times \Delta^1}$};
	\node (b) at (0,0) {$\tilde{A}$};
	\node (c) at (2,1) {$\stack{X}$};
	\node (d) at (2,0) {$\stack{Y}$};
	
	\draw[->] (a) to node[left] {$\widetilde{\pr}_1$} (b);
	\draw[->] (a) to node[above] {$H$} (c);
	\draw[->] (b) to node[below] {$p \circ f$} (d);
	\draw[->] (c) to node[right] {$p$} (d);
	
	\draw[dbl] (0,0.5) to [bend right=35] node[below right] {$\psi$} (1,1);
	
	\end{tikzpicture}
		\]
		 such that $\psi_0  = p \circ \epsilon_0$ and
		  $\psi_1 \cdot (p \circ \epsilon_1) = \varphi$.
	\end{itemize}
(Notation: $H_0:=H\circ\tilde{i}$, where $i : A \to A\times\Delta^1$ is the time $0$ map.) 
In the case where $\varphi$ and $\psi$ are both identity 2-isomorphisms 
(so $p \circ f=p \circ g$ and $p\circ f \circ \widetilde{\pr}_1 = p \circ H$)
we say that $H$ is a \emph{restricted homotopy relative to $\stack{Y}$}.

A restricted fiberwise homotopy as above is called \emph{strict} if $\epsilon_0$   
and $\epsilon_1$ are the identity 2-isomorphisms. 
\end{defn}

\begin{rem}
\label{R:fake_remark}
In view of the adjunction of \cref{L:restriction isom}, we can replace the 
diagrams above with their corresponding diagram in the category of simplicial 
groupoids. For example,
	\[
\begin{tikzpicture}[scale=1.75]

	\node (a) at (0,0) {$A$};
	\node (b) at (1,1) {$\stack{X}\Del$};
	\node (c) at (1,0) {$\stack{Y}\Del$};
	
	\draw[->] (a.north) to [bend left = 20] node[above] {\scriptsize $f'$} (b.west);
	\draw[->] (a) to [bend right = 0] node[above] {\scriptsize $g'$} (b);
	\draw[->] (a.10) to [bend left = 15] node[above,pos=0.6,inner sep=1pt] (pf)
	{\scriptsize $p' \circ f'$} (c.170);
	\draw[->] (a.south east) to [bend right = 20] node[below,inner sep=1pt] (pg)
	{\scriptsize $p' \circ g'$} (c.south west);
	\draw[->] (b) to node[right] {\scriptsize $p'$} (c);
	
	\draw[double,
              double equal sign distance,
              -implies,
              shorten >= 2pt,
              shorten <= 2pt
              ]
              (pf) to node[left] {$\varphi'$} (pg);
	
\end{tikzpicture}
\]
Thus, we can regard a restricted homotopy as a homotopy in the category of simplicial groupoids.
\end{rem}

\begin{rem}{\label{R:restrictedfiberwisehomotopy}}
As in \cref{R:fiberwisehomotopy}, there is some 
 flexibility in choosing $H$, namely, we are allowed to replace $H$ by any map 2-isomorphic
 to it (and adjust $\epsilon_0$, $\epsilon_1$ and $\psi$ accordingly).
\end{rem}

An ordinary homotopy gives rise to a restricted homotopy.

\begin{lemma}
\label{L:htpy=>fakehtpy}
 Let $A$ be a simplicial set and let $\stack{A}:=\tilde{A}$. Notation being as 
 in \cref{D:stack_homotopy}, suppose that we are given a fiberwise homotopy 
 $(H,\epsilon_0,\epsilon_1,\psi)$ from $f$ to $g$ relative to $\varphi$. Then, 
 precomposing with the natural map $\widetilde{A\times\Delta^1} \to \tilde{A}\times [0,1]$ 
 gives rise to a restricted fiberwise homotopy from   $f$ to $g$ relative to $\varphi$.
\end{lemma}

\begin{proof}
Straightforward.
\end{proof}

\section{Lifting conditions}
\label{S:Lemmas}

We shall review some of the material from \cite{No14} and recall the notion of 
(weak) Serre fibration between stacks. For a full account see Sections 2 and 3 of \cite{No14}. 
Before we start, it is worthwhile to emphasize the difference between the notion of fibration  
in this section and the standard ones in well known model category structures on the 
category of presheaves of groupoids: our notion is more geometric, 
in the sense that it does not distinguish between equivalent presheaves; 
in particular, any equivalence of presheaves of groupoids is a fibration in our sense.


\begin{defn}
\label{D:WLLP}
Let $i:\stack{A} \to \stack{B}$ and $p:\stack{X} \to \stack{Y}$ be morphisms of 
presheaves of groupoids. 
Then, $i$ has the \emph{weak left lifting property (WLLP)} with respect to $p$ if given
	\[
\begin{tikzpicture}[scale=1]
	\node (a) at (0,2) {$\stack{A}$};
	\node (b) at (0,0) {$\stack{B}$};
	\node (c) at (2,2) {$\stack{X}$};
	\node (d) at (2,0) {$\stack{Y}$};
	
	\draw[->] (a) to node[left] {$i$} (b);
	\draw[->] (a) to node[above] {$f$} (c);
	\draw[->] (b) to node[below] {$g$} (d);
	\draw[->] (c) to node[right] {$p$} (d);
	
	\draw[dbl] (1,2) to [bend left=30] node[below right] {$\alpha$} (0,1);
	
\end{tikzpicture}
\]
	there is a morphism $h:\stack{B} \to \stack{X}$, a 2-isomorphism
	$\gamma:g \Rightarrow p \circ h$ and a fiberwise homotopy $H$
	from $f$ to $h \circ i$ relative to $\alpha\cdot(\gamma \circ i)$:
	\[
\begin{tikzpicture}[scale=1]
	\node (a) at (0,2) {$\stack{A}$};
	\node (b) at (0,0) {$\stack{B}$};
	\node (c) at (2,2) {$\stack{X}$};
	\node (d) at (2,0) {$\stack{Y}$};
	
	\draw[->] (a) to node[left] {$i$} (b);
	\draw[->] (a) to node[above] {$f$} (c);
	\draw[->] (b) to node[below] {$g$} (d);
	\draw[->] (c) to node[right] {$p$} (d);
	
	\draw[->,dashed] (b) to node[above] {$h$} (c);
	
	\draw[dbl] (1,0) to [bend right=0] node[below right,inner sep=1pt] {$\gamma$} (2,1);
	
	\node at (0.45,1.55) {$H$};
\end{tikzpicture}
\]
We say that $i$ has the \emph{left lifting property (LLP)} with respect to $p$ 	 
if $H$ can be taken to be a ghost homotopy. In other words, there are 2-isomorphisms 	
$\beta:f \Rightarrow h \circ i$ and $\gamma:g \Rightarrow p \circ h$ such that
$p\circ\beta=\alpha\cdot(\gamma \circ i)$, i.e., the following
diagram commutes ($\alpha$ is not shown in the diagram):
	\[
	\begin{tikzpicture}[scale=1]
	\node (a) at (0,2) {$\stack{A}$};
	\node (b) at (0,0) {$\stack{B}$};
	\node (c) at (2,2) {$\stack{X}$};
	\node (d) at (2,0) {$\stack{Y}$};
	
	\draw[->] (a) to node[left] {$i$} (b);
	\draw[->] (a) to node[above] {$f$} (c);
	\draw[->] (b) to node[below] {$g$} (d);
	\draw[->] (c) to node[right] {$p$} (d);
	
	\draw[->,dashed] (b) to node[above] {$h$} (c);
	
	\draw[dbl] (1,0) to [bend right=0] node[below right,inner sep=1pt] {$\gamma$} (2,1);
	\draw[dbl] (1,2) to [bend right=0] node[above left,inner sep=1pt] {$\beta$} (0,1);

\end{tikzpicture}
\]
We say that $p$ has the \emph{(weak) covering homotopy property} with respect to 
$\stack{A}$, if the
inclusion $\stack{A} \to  \stack{A}\times [0,1]$, $a \mapsto (a,0)$, has  
(W)LLP with respect to $p$. 
\end{defn}

\begin{rem}
\label{R:weak}
The usage of the term `weak' (which means, `up to fiberwise homotopy') in the above definition 
is in conflict  with our usual usage of the term weak 
(which means, `up to 2-isomorphism', as opposed to `strict'). 
But since the above definition is quite standard in the homotopy theory literature, 
we deemed it inappropriate to change it. We apologize for the confusion this may cause.
\end{rem}

\begin{defn}[\cite{No14}, Definitions 3.6, 3.7]
\label{D:WTS_fib}
A  morphism of presheaves of groupoids $p:\stack{X} \to \stack{Y}$ is called a 
\emph{(weak) Serre fibration} if it has the 
(weak) covering homotopy property with respect to every finite CW complex $A$.
That is, $A \to A \stimes [0,1]$ has the (W)LLP with respect to $p$.
It is called a \emph{(weak) trivial Serre fibration} if every finite CW inclusion 
$i:A \hookrightarrow B$ has the (W)LLP with respect to $p$.
\end{defn}


\begin{lemma}[see \cite{No14}, Proposition 3.21]{\label{L:SerrevsSerre}}
Let $p:\stack{X} \to \stack{Y}$ be a (weak) Serre fibration.
Then,  every cellular inclusion $i:A \hookrightarrow B$
of finite CW complexes that induces isomorphisms on all $\pi_n$   has the (W)LLP 
with respect to $p$.
\end{lemma}

\begin{proof}
  The map $i : A \hookrightarrow B$ being as above, 
  $A$ becomes a deformation retract of $B$. Therefore,
  the map $i$ is a retract of the the map $j : B \to B \times [0,1]$, $j(b)=(b,0)$,
    \[
	\begin{tikzpicture}[scale=1]
	\node (a) at (0,2) {$A$};
	\node (b) at (0,0) {$B$};
	\node (c) at (2,2) {$B$};
	\node (d) at (2,0) {$B \times [0,1]$};
	
	\draw[->] (a) to node[left] {$i$} (b);
	\draw[->] (c) to node[above] {$r$} (a);
	\draw[->] (d) to node[below] {$H$} (b);
	\draw[->] (c) to node[right] {$j$} (d);
	
\end{tikzpicture}
\]
 Here, $r$ is the retraction and $H :  B \times [0,1] \to B$ is a homotopy with
 $H_0=r$ and $H_1=\id_B$.
 Since $j$ has (W)LLP with respect to $p$, so does its retract $i$.
\end{proof}

\begin{rem}
\label{R:serre}
As opposed to the notion of Reedy fibration that we will introduce in \cref{D:reedy_cond}, 
the notion of (weak) 	Serre fibration is ``intrinsic'' (or ``geometric'') in the sense 
that 	if $p:\stack{X} \to \stack{Y}$ is a (weak) 	Serre fibration and 
$p':\stack{X}' \to \stack{Y}'$ is a morphism equivalent to it, 	then $p'$ 
is also a (weak)	Serre fibration.
\end{rem}

\begin{prop}
\label{P:trivserre}
Let $p :\stack{X} \to \stack{Y}$ be a morphism of topological stacks, and assume 
that $\stack{X}$ is  Serre. Then, $p$ is a (weak) trivial Serre fibration if and only 
if it is a (weak) Serre fibration and a weak equivalence.
\end{prop}

\begin{proof}
By (\cite{No14}, Lemma 2.4), every morphism $p :\stack{X} \to \stack{Y}$
of topological stacks with $\stack{X}$ a Serre stack is a
Serre morphism (in the sense of \cite{No14}, Definition 2.2).
The result now follows from (\cite{No14}, Proposition 5.4). Note that 
(\cite{No14}, Proposition 5.4) is only stated for trivial Serre fibration, but it is also true
for trivial weak  Serre fibration; the first paragraph of the given proof 
(minus the last sentence) is in fact the proof of the statement for trivial weak  
Serre fibration.
\end{proof}

\section{Reedy fibrations of stacks}
\label{S:ModelStr}

In this section, we introduce Reedy fibrations between presheaves
of groupoids (\cref{D:reedy_cond}) and establish some of their basic
properties.

\subsection{Model structure on $\gpd$}
\label{SS:modelgpd}

Let $\gpd$ denote the 2-category of groupoids. In this section we discuss the model
structure on the underlying 1-category of $\gpd$.

\begin{defn}
\label{D:gpd_fib}
Let $p: G \to H$ be a morphism in $\gpd$. We say that $p$ is a \emph{fibration} if for any 
$x \in G$ and any isomorphism 	$\varphi :y \to p(x)$ in $H$, there exists an isomorphism 
$\psi:z \to x$ in $G$ such that $p(\psi)=\varphi$.
In the literature, this is commonly referred to as an \emph{isofibration}.
\end{defn}

There is a model category structure on the category $\gpd$ of groupoids where
	\begin{itemize}
		\item weak equivalences are equivalences of groupoids;
		\item cofibrations are maps that are  injective on the set of objects;
		\item fibrations are  as in \cref{D:gpd_fib}.
	\end{itemize}
We refer the reader to (\cite{Hollander08}, Theorem 2.1) for more detail
and further references.

\begin{lemma}
\label{L:strictvs2-categorical}  
Consider the following diagram in $\gpd$:
\[
\begin{tikzpicture}[scale=1]

	\node (b) at (0,0) {$H$};
	\node (c) at (2,2) {$K$};
	\node (d) at (2,0) {$G$};
	
	\draw[->] (b) to node[below] {} (d);
	\draw[->] (c) to node[right] {p} (d);

\end{tikzpicture}
\]
Suppose that $p$ is a fibration. Then, the 
natural map of groupoids
 	 \[ H\times_{G}K \to H\ttimes_{G}K
	 \]
is an equivalence.	 
\end{lemma}

\begin{proof}
  This functor is always fully faithful (see \cref{SS:strict}). It
  is straightforward that fibrancy of
  $p$ implies essential surjectivity.
\end{proof}

\begin{lemma}
\label{L:trivcof}
A morphism $i : G \to H$ in $\gpd$ is a trivial cofibration  if and only if it 
is essentially surjective and induces an isomorphism of groupoids between $G$ 
and a full subcategory of $H$. When this is the case, $G \times K \to H\times K$ is  
a trivial cofibration for every groupoid $K$.
\end{lemma}

\begin{proof}
Straightforward.
\end{proof}

\begin{prop}
\label{P:gpd}
The above model structure on $\gpd$ is left proper, simplicial, cofibrantly 
generated, combinatorial and monoidal (with respect to cartesian product).
\end{prop}

\begin{proof}
The properties left proper, simplicial and cofibrantly generated are proved in
(\cite{Hollander08}, Theorem 2.1). Since $\gpd$  is cofibrantly generated and
locally presentable, it is, by  definition, combinatorial.

To check that the model structure is monoidal we need to verify conditions ($i$)-($iii$)
of (\cite{HTT}, Definition A.3.1.2). Conditions ($ii$) and ($iii$) are obvious. To check
($i$) we have to show that the cartesian product $\times : \gpd\times \gpd \to \gpd$ is
a left Quillen bifunctor. That is, the following two conditions are satisfied:
\begin{itemize}
 \item[(a)] Let $i : A \to A'$ and $j : B \to B'$ be cofibrations in $\gpd$. Then,
      the induced map
    \[i\wedge j : (A'\times B) \underset{A\times B}{\coprod} (A\times B') \to A'\times B'
     \]
      is a cofibration in $\gpd$. Moreover, if either $i$ or $j$ is a trivial cofibration, 
      then $i\wedge j$ is also a trivial cofibration.
 \item[(b)] The cartesian product preserves small colimits separately in each variable.
\end{itemize}
The first part of (a) is easy as it only concerns the object sets of the groupoids in
question, and the corresponding statement is true in the category of sets. To prove the
second part of (a), assume that $i : A \to A'$ is a trivial cofibration.
The claim follows from \cref{L:trivcof} and two-out-of-three applied to
\[ A\times B' \to  (A'\times B) \underset{A\times B}{\coprod} (A\times B') \to A'\times B'.
\]
Condition (b)  can be checked as follows (this argument was suggested to us
by the referee). Let $K$ be an arbitrary groupoid. We have
 \begin{align*}
    \Hom_{\gpd}(\colim_{\alpha}(G_{\alpha} \times H), K)   
    & \cong  \lim_{\alpha} \Hom_{\gpd}(G_{\alpha} \times H, K) \\
    & \cong   \lim_{\alpha} \Hom_{\gpd}(G_{\alpha}, \Hom_{\gpd}(H, K))   \\ 
    & \cong  \Hom_{\gpd}(\colim_{\alpha} G_{\alpha}, \Hom_{\gpd}(H, K)) \\ 
    & \cong  \Hom_{\gpd}((\colim_{\alpha} G_{\alpha}) \times H, K).
\end{align*}
Thus, $\colim_{\alpha}(G_{\alpha}) \times H\cong (\colim_{\alpha} G_{\alpha}) \times H$.
\end{proof}

\begin{prop}
\label{P:excellent}
The model structure on $\gpd$ is excellent in
the sense of (\cite{HTT}, Definition A.3.2.16).
\end{prop}

\begin{proof}
Axioms (A1)-(A4) of [ibid.] are straightforward to check. Axiom (A5), the 
Invertibility Hypothesis, follows from (\cite{HTT}, Lemma A.3.2.20) applied to 
the fundamental groupoid functor $\Pi_1 : \sset \to \gpd $.
\end{proof}

\subsection{Injective model structure on $[\category{C}^{op},\gpd]$}
\label{SS:injective}

Let $\category{C}$ be a small category. Since the model structure on $\gpd$ is 
combinatorial (\cref{P:gpd}), by (\cite{HTT}, Proposition A.2.8.2) there is a  
model structure on the category $[\category{C}^{\text{op}},\gpd]$ of presheaves 
of groupoids, called the \emph{injective model structure}, where
	\begin{itemize}
		\item weak equivalences are the objectwise weak equivalences as in \cref{SS:modelgpd};
		\item cofibrations are the objectwise cofibrations as in \cref{SS:modelgpd};
		\item fibrations have the right lifting property with respect to the trivial
		 cofibrations.
	\end{itemize}
We refer the reader to (\cite{HTT}, A.2.8) for more details on the injective
model structure.

\begin{prop}
\label{P:enriched}
The injective model structure on $[\category{C}^{\text{op}},\gpd]$ is $\gpd$-enriched
in the sense of (\cite{HTT}, Definition A.3.1.5).
\end{prop}

\begin{proof}
 This follows from (\cite{HTT}, Remark A.3.3.4).
\end{proof}

We are particularly interested in the cases $\category{C}=\topspace$  and 
$\category{C}=\Delta$. In the case $\category{C}=\Delta$,
we have an explicit description of fibrations thanks to \cref{P:reedy=inj} below.

\begin{rem}{\label{R:universe}}
 The smallness assumption on $\category{C}$ is to allow us to quote results from
 Appendix A of
 \cite{HTT}. As indicated at the beginning of Appendix A of \cite{HTT}, this is not
 a restrictive assumption as we can always fix a Grothendieck universe. For this reason,
 our treating $\category{C}=\topspace$ as a small category is not problematic.
\end{rem}

\subsection{Reedy model structure on $\sgpd$}
\label{SS:reedy}

The Reedy model  structure on the category of simplicial groupoids
$\sgpd=[\Deltacat^{\text{op}},\gpd]$ is defined as follows:
  	\begin{itemize}
		\item weak equivalences are the objectwise weak equivalences;
		\item cofibrations are morphisms $X \to Y$ such that for every $n$ the map
	        \begin{align*}
		          L_nY \coprod_{L_nX} X_n \to Y_n
	        \end{align*}
	       is a cofibration of groupoids (as in \cref{SS:modelgpd});
		\item fibrations are morphisms  $X \to Y$ such that for every  $n$ the map
	      \begin{align*}
		    X_n \to M_nX \times_{M_nY} Y_n
	      \end{align*}
	        is a fibration of groupoids (as in \cref{SS:modelgpd}).
	\end{itemize}

Here, $L_nX$ stands for the latching object
	\begin{align*}
		L_nX:= \displaystyle \colim_{[k]\underset{\neq}{\twoheadrightarrow}[n]} X_k,
	\end{align*}
and $M_nX$ stands for the	matching object
	\begin{align*}
		M_nX:=\displaystyle \lim_{[n]\underset{\neq}{\hookrightarrow}[k]} X_k.
	\end{align*}

Let us unravel the above definitions. First of all, recall that the matching object
$M_nX$ can be alternatively described by
   \[M_nX=\Hom_{\sgpd}(\partial\Delta^n,X),
   \]
where we regard the simplicial set $\partial\Delta^n$ as a simplicial groupoid.
The map $X_n \to M_nX$ is the
one induced by the inclusion $\partial\Delta^n \hookrightarrow \Delta^n$.

The Reedy fibration condition can now be restated as saying that
  \[\Hom_{\sgpd}(\Delta^n,X)  \to
   \Hom_{\sgpd}(\partial\Delta^n,X) \times_{\Hom_{\sgpd}(\partial\Delta^n,Y)}
  \Hom_{\sgpd}(\Delta^n,Y)
  \]
is a fibration of groupoids.

\begin{lemma}
\label{L:reedyforsets}
Let $X$ and $Y$ be simplicial sets, regarded as objects in $\sgpd$. Then, any morphism
$p : X \to Y$ is a Reedy fibration.
\end{lemma}

\begin{proof}
This follows from the definition of a Reedy fibration and the fact that every map
of sets, regarded as objects in $\gpd$, is a fibration of groupoids.
\end{proof}

The Reedy cofibrations turn out to coincide with the  objectwise cofibrations. 
That is, a morphism $p : X \to Y$ of simplicial groupoids is a Reedy cofibration 
if and only if  $X_n \to Y_n$ is a cofibration of groupoids (in the sense of 
\cref{SS:modelgpd}) for all $n$. This is a consequence of the following proposition.

\begin{prop}
\label{P:reedy=inj}
	The Reedy model structure and the injective model structure on
	$\sgpd=[\Deltacat^{\text{op}},\gpd]$ coincide.
\end{prop}

\begin{proof}
We know that, by definition, the two model structures have the same weak equivalences. 	 
It remains to show that they have the same fibrations. Let 
$N : \sgpd=[\Deltacat^{\text{op}},\gpd] \to [\Deltacat^{\text{op}},\sset]$ 	
be the objectwise nerve functor, and let $\Pi_1$ be its left adjoint, the objectwise
fundamental groupoid functor. We show that the following are equivalent:
	
	\begin{enumerate}
		\item[(1)]   $p:X \to Y$ is a Reedy fibration in $\sgpd$.
		\item[(2)]   For all $n$, $N(X_n) \to M_n(N(X))\times_{M_n(N(Y))} N(Y_n)$
		is a Kan fibration of simplicial sets.
		\item[(3)]   $N(p): N(X) \to N(Y)$ is a Reedy fibration in
		$[\Deltacat^{\text{op}},\sset]$.
		\item[(4)]   $N(p):N(X) \to N(Y)$ is an injective fibration in
		$[\Deltacat^{\text{op}},\sset]$.
		\item[(5)]   $p:X \to Y$ is an injective fibration in
		$[\Deltacat^{\text{op}},\gpd]=\sgpd$.
	\end{enumerate}

(1) $\Leftrightarrow$ (2) is true since the nerve functor
preserves fiber products and $G \to H$ is a fibration of groupoids if and only
if $N(G) \to N(H)$ is a Kan fibration. (2) $\Leftrightarrow$ (3) is true by definition.
(3) $\Leftrightarrow$ (4)
follows from the fact that injective model structure on $[\Deltacat^{\text{op}},\sset]$
is the same as the Reedy model structure (\cite{HTT}, Example A.2.9.8 and Example A.2.9.21).

The implication (5) $\Rightarrow$ (4) follows from the fact that $\Pi_1$
takes a trivial cofibration of (presheaves of) simplicial sets
to a trivial cofibration of (presheaves of) groupoids.

Finally,  to prove (4) $\Rightarrow$ (5)  we use the fact that $N$ preserves
trivial cofibrations (for this use \cref{L:trivcof}) and that
$\Pi_1\circ N=\text{id}_{\sgpd}$.
More precisely, to solve a lifting problem in $[\Deltacat^{\text{op}},\gpd]$,
we can first apply $N$, solve the lifting problem
in $[\Deltacat^{\text{op}},\sset]$, and then apply $\Pi_1$ to obtain a solution to the
original lifting problem.
\end{proof}

\subsection{Reedy fibrations in $\presheafgpd$}
\label{SS:reedystack}
From now on, $\category{C}=\topspace$ (see \cref{R:universe}). We will use the notation
$\presheafgpd$ instead of $[\category{C}^{\text{op}},\gpd]$.
We begin with our main definition.

\begin{defn}
\label{D:reedy_cond}
We say that a map of presheaves of groupoids $p:\stack{X} \to \stack{Y}$ 
is a \emph{Reedy fibration} if $p\Del:\stack{X}\Del \to \stack{Y}\Del$ 
is a Reedy fibration in $\sgpd$ (see \cref{SS:reedy}).
\end{defn}

\begin{lemma}
\label{L:reedyforsets2}
Let $X$ and $Y$ be presheaves of simplicial sets, regarded as objects in $\presheafgpd$. 
Then, any morphism $p : X \to Y$ is a Reedy fibration.
\end{lemma}

\begin{proof}
This follows from \cref{L:reedyforsets}.
\end{proof}

\begin{prop}
\label{P:inj=>reedy}
If $p:\stack{X} \to \stack{Y}$ is an injective fibration of presheaves of groupoids, 
then $p$ is a Reedy fibration.
\end{prop}

\begin{proof}
      We have to show that, for every $n$, the map
      	\begin{align*}
		\Hom_{\sgpd}(\Delta^n,\stack{X}\Del) \to
		 \Hom_{\sgpd}(\partial\Delta^n,\stack{X}\Del)
		 \times_{\Hom_{\sgpd}(\partial\Delta^n,\stack{Y}\Del)}
		  \Hom_{\sgpd}(\Delta^n,\stack{Y}\Del)
	\end{align*}
 is a fibration of groupoids.   Via the tilde construction,
 the above map is isomorphic to
           	\begin{align*}
		\Hom_{\presheafgpd}(\widetilde{\Delta^n},\stack{X}) \to
		 \Hom_{\presheafgpd}(\widetilde{\partial\Delta^n},\stack{X})
		 \times_{\Hom_{\presheafgpd}(\widetilde{\partial\Delta^n},\stack{Y})}
		  \Hom_{\presheafgpd}(\widetilde{\Delta^n},\stack{Y}).
	\end{align*}
This map is a fibration of groupoids because $p:\stack{X} \to \stack{Y}$ is a fibration 
and $\widetilde{\partial\Delta^n} \to \widetilde{\Delta^n}=\Delta^n$ is a cofibration 	 
in the injective model structure on $\presheafgpd$ (to see the latter, write 	 
$\partial\Delta^n$ as the colimit of its faces 	 and use the fact that the tilde 
construction preserves colimits). The claim now follows from
\cref{P:enriched} (also see \cite{HTT}, Remark A.3.1.6(2')).
\end{proof}

\begin{prop}
\label{P:reedy cond2}
  Let $p: \stack{X} \to \stack{Y}$ be a Reedy fibration of presheaves of groupoids,  and
  let $A \to B$ be a monomorphism of simplicial sets. Then, the map
	\begin{align*}
		\Hom_{\sgpd}(B,\stack{X}\Del) \to
		 \Hom_{\sgpd}(A,\stack{X}\Del)
		 \times_{\Hom_{\sgpd}(A,\stack{Y}\Del)}
		  \Hom_{\sgpd}(B,\stack{Y}\Del)
	\end{align*}
 and, equivalently (see \cref{L:restriction isom}), the map
	 \begin{align*}
		\Hom_{\presheafgpd}(\tilde{B},\stack{X}) \to
		 \Hom_{\presheafgpd}(\tilde{A},\stack{X})
		 \times_{\Hom_{\presheafgpd}(\tilde{A},\stack{Y})}
		  \Hom_{\presheafgpd}(\tilde{B},\stack{Y})
     \end{align*}
	are fibrations of groupoids.
\end{prop}

\begin{proof}
In fact, the first map is a fibration of groupoids for any Reedy fibration 	
$X \to Y$ in $\sgpd$ (in our case $X=\stack{X}\Del$ and $Y=\stack{Y}\Del$). 
In view 	of \cref{P:reedy=inj} this follows from \cref{P:enriched} with
$\category{C}=\Delta$ (also see \cite{HTT}, Remark A.3.1.6(2')).
	
Alternatively, use (\cite{Dugger01}, Lemma 4.5), with $M=\gpd$, $K=A$, $L=B$,
$X=\stack{X}\Del$ and $Y=\stack{Y}\Del$.
\end{proof}

\begin{prop}
\label{P:reedy_replace}
	For any morphism of presheaves of groupoids $p:\stack{X} \to \stack{Y}$,
	there exists a strictly commutative diagram
	\[
	\begin{tikzpicture}[scale=1]

	\node (a) at (0,1.5) {$\stack{X}$};
	\node (b) at (0,0) {$\stack{X}'$};
	\node (c) at (2,1.5) {$\stack{Y}$};
	
	\draw[->] (a) to node[left] {$\sim$} node[right] {$g$} (b);
	\draw[->] (a) to node[above] {$p$} (c);
	\draw[->] (b) to node[below right] {$p'$} (c);

\end{tikzpicture}
\]
where $p'$ is an injective (hence, also Reedy) fibration and $g: \stack{X} \risom\stack{X}'$ 
is an equivalence of presheaves of groupoids.
\end{prop}

\begin{proof}
	Take the usual fibrant replacement in the injective model structure on
	$\presheafgpd$ and use \cref{P:inj=>reedy}.
\end{proof}

\section{Singular functor for stacks}
\label{S:Sing}
We shall define the  functor $\sing:\presheafgpd \to \sset$ of singular chains
and establish some of its basic properties. This functor will be the focus of
the rest of the paper.

\subsection{The functors $B$ and $\sing$}

\begin{defn}
\label{D:sing}
	Let
\begin{align*}
	B:\presheafgpd &\to \bsset,\\
	\stack{X} &\mapsto N(\stack{X} \Del ),
\end{align*}
\begin{align*}
	\sing:\presheafgpd &\to \sset,\\
	\stack{X} &\mapsto \Diag(N(\stack{X} \Del)).
\end{align*}
Here, $\bsset$ stands for the category of bisimplicial sets,
$N:\sgpd \to \bsset$ is the levelwise 
nerve functor, and
$\Diag:\bsset \to \sset$ refers to taking the diagonal of a bisimplicial set.
\end{defn}

\begin{rem}
\label{R:sing restricts}
When restricted to $\topspace$, the functor $\sing$ coincides with the usual singular 
chains functor $\sing : \topspace \to \sset$. More precisely, the following diagram 
commutes:
	\[
	\begin{tikzpicture}[scale=1]

	\node (a) at (0,1.5) {$\topspace$};
	\node (b) at (2,0) {$\sset$};
	\node (c) at (2,1.5) {$\presheafgpd$};
	
	\draw[->] (a) to node[below left] {$\sing$} (b);
	\draw[->] (a) to node[above] {} (c);
	\draw[->] (c) to node[right] {$\sing$} (b);

\end{tikzpicture}
\]
The top arrow in this diagram is the Yoneda embedding.
\end{rem}

\begin{lemma}
\label{L:levelequiv=>we}
Let $f : X \to Y$ be a morphism of simplicial groupoids that induces  equivalences of 
groupoids $X_n \to Y_n$ for all $n$. Then, the induced map $\Diag(NX) \to \Diag(NY)$ is
a weak equivalence of simplicial sets.
\end{lemma}

\begin{proof}
	This follows from (\cite{Jardine}, Chapter IV, Proposition 1.7).
\end{proof}

\begin{cor}
\label{C:equiv=>we}
Let $f : \stack{X} \to \stack{Y}$ be an equivalence of presheaves of groupoids. Then, 
$\sing(f) :\sing(\stack{X}) \to \sing(\stack{Y})$ is a weak equivalence of simplicial sets.
\end{cor}

We will need the following definition  from (\cite{Jardine}, Chapter IV, Section 3.3).

\begin{defn}
\label{D:diag_adjoint}
Define the functor
\[
	d^*:\sset \to \bsset
\]
to be the one  uniquely determined by the following two properties:
\begin{itemize}
	\item $d^*(\Delta^n)=\Delta^{n,n}$ (see \cref{Setup} for notation);
	\item $d^*$ preserves colimits.
\end{itemize}
\end{defn}

\begin{prop}
\label{P:Adjunction}
	The functors $\Diag:\bsset \to \sset$ and $N:\sgpd \to \bsset$ have left adjoints:
\[
\begin{tikzpicture}
	
	\node[inner sep=4pt] (a) at (0,0) {$\sgpd$};
	\node (b) at (2,0) {$\bsset$};
	\node (c) at (4,0) {$\sset$.};
	
	\draw[->,transform canvas={yshift=0.7ex}] (a) to node[above] {$N$} (b);
	\draw[->,transform canvas={yshift=0.7ex}] (b) to node[above] {$\Diag$} (c);
	
	\draw[->,transform canvas={yshift=-0.7ex}] (b) to node[below] {$\Pi_1$} (a);
	\draw[->,transform canvas={yshift=-0.7ex}] (c) to node[below] {$d^*$} (b);
		
\end{tikzpicture}
\]
Here, $\Pi_1$ denotes the fundamental groupoid functor, and $d^*$ is as 
in (\cref{D:diag_adjoint}). Therefore, $\Diag\circ N$ also has $\Pi_1 \circ d^*$ 
as left adjoint. In particular, the functors $N$, $\Diag$ and 
$\sing=\Diag\circ N\circ ()\Del$ 
preserve limits.
\end{prop}

\begin{proof}
For the first adjunction see (\cite{Hollander08}, Corollary 2.3).
The second adjunction is  discussed in (\cite{Jardine}, Chapter IV, Section 3.3).
\end{proof}

\begin{lemma}
\label{L:respectshomotopy} Let $f,g : \stack{X} \to \stack{Y}$
be morphisms of presheaves of groupoids.
\begin{itemize}
\item[(i)]
If $\alpha : f \Rightarrow g$ is a 2-isomorphism, then
we have an induced homotopy $\hat{\alpha}$
from $\sing(f)$ to $\sing(g)$.
\item[(ii)] If $h$ is a strict homotopy from $f$ to $g$
(see  \cref{D:stack_homotopy}), then we have an induced homotopy $\hat{h}$
from $\sing(f)$ to $\sing(g)$.
\end{itemize}
\end{lemma}

\begin{proof}
Part ($ii$) follows from the fact that $\sing$ commutes with products (\cref{P:Adjunction}). 
To prove part ($i$), let $\mathcal{I}$ be the constant presheaf of categories 
$\mathcal{I}: T \mapsto \{0\to 1\}$, where $\{0\to 1\}$ is the ordinal category 
(also denoted $[1]$). A 2-isomorphism $\alpha$ as above is the same thing as a morphism
\[\Phi_{\alpha} : \stack{X}\times\mathcal{I} \to \stack{Y}\]
whose restrictions to $\{0\}$ and $\{1\}$ are $f$ and $g$, respectively.  
It is easy to see that $\sing(\mathcal{I})=\Delta^1$.  (Note that we have only 
defined $\sing$ for presheaves of groupoids, but clearly the same definition makes 
sense for presheaves of categories as well.) By \cref{P:Adjunction}, we obtain a 
map of simplicial sets
   \[\hat{\alpha}:=\sing(\Phi_{\alpha}) : \sing(\stack{X})\times\Delta^1 \to \sing(\stack{Y}).
   \]
This is the desired homotopy.
\end{proof}

\begin{rem} $ $
\begin{itemize}
\item[(1)] The operation $\alpha \mapsto \hat{\alpha}$ respects composition of 2-isomorphisms
 in the sense that $\widehat{\alpha\cdot\beta}$ is canonically homotopic to the ``composition''
 of $\hat{\alpha}$ and $\hat{\beta}$. More precisely,
 $\hat{\alpha}$, $\hat{\beta}$ and $\widehat{\alpha\cdot\beta}$
 are the three faces of a canonical map
  \[\sing(\stack{X})\times\Delta^2 \to \sing(\stack{Y}).\]
 We also have higher coherences. That is, every string of $k$ composable 2-isomorphisms
 defines a canonical map   \[\sing(\stack{X})\times\Delta^k \to \sing(\stack{Y})\]
 whose restriction to various faces represent different ways of composing (a subset) of
 homotopies associated to these 2-isomorphisms.

\item[(2)] In the statement of \cref{L:respectshomotopy}($ii$) we could use
 a general homotopy $h=(H,\epsilon_0,\epsilon_1)$ from $f$ to $g$
(see  \cref{D:stack_homotopy}), but in this case instead of a homotopy from $f$ to $g$ we obtain
a sequence of three composable homotopies $\hat{\epsilon_0}$, $\hat{\epsilon_1}$ and $\hat{h}$.
\end{itemize}
\end{rem}

\begin{ex} In \cref{L:2fiberprod} below we will discuss the effect of $\sing$ on
2-fiber products of presheaves of groupoids. To motivate the assumptions made there,
we look at the following examples.
\begin{itemize}
\item[(1)]
 The functor $\sing$ does not respect 2-fiber products. For example,
 let $\stack{Z}$ be the constant presheaf on $\topspace$ with value $J$ (viewed as a stack), 
 where $J=\{0\longleftrightarrow 1\}$ is the interval groupoid, and let 
 $\stack{X}=\stack{Y}=*$ be singletons mapping to the points $0$ and $1$ 
 in $\stack{Z}$, respectively.
 Then,
 \[\stack{X}\ttimes_{\stack{Z}}\stack{Y}=*\ttimes_{\stack{Z}}*\]
 is equivalent to a point, while
  \[\sing(\stack{X})\times_{\sing(\stack{Z})}\sing(\stack{Y}) =
  *\times_{\sing(\stack{Z})}*\]
  is the empty set.
\item[(2)] It is not reasonable to expect that $\sing$ takes 2-fiber products to
homotopy fiber products either. For example, let $\stack{Z}=[0,1]$ be the unit
interval, and let $\stack{X}=\stack{Y}=*$ be singletons
 mapping to  the points $0$ and $1$ in $\stack{Z}$, respectively.
  Then,
 \[\stack{X}\ttimes_{\stack{Z}}\stack{Y}=*\ttimes_{\stack{Z}}*=*\times_{[0,1]}*\]
 is the empty set, while
  \[\sing(\stack{X})\stackrel{h}{\times}_{\sing(\stack{Z})}\sing(\stack{Y}) =
  *\stackrel{h}{\times}_{\sing(\stack{Z})}*\]
  is non-empty (in fact, homotopy equivalent to a point).
\end{itemize}
\end{ex}

\begin{lemma}
\label{L:2fiberprod}
Consider the following diagram in $\presheafgpd$:
\[
\begin{tikzpicture}[scale=1]

	\node (b) at (0,0) {$\stack{Y}$};
	\node (c) at (2,2) {$\stack{X}$};
	\node (d) at (2,0) {$\stack{Z}$};
	
	\draw[->] (b) to node[below] {} (d);
	\draw[->] (c) to node[right] {p} (d);

\end{tikzpicture}
\]
Suppose that $p$ is a Reedy fibration (by \cref{L:reedyforsets2}
this is automatic if $\stack{X}$ and $\stack{Z}$
are presheaves of sets). Then, there is a natural weak equivalence of simplicial
sets
 \[\sing(\stack{X})\times_{\sing(\stack{Z})}\sing(\stack{Y})
 \stackrel{\sim}{\longrightarrow} \sing(\stack{X}\ttimes_{\stack{Z}}\stack{Y}).
 \]
\end{lemma}

\begin{proof}
	Since $p$ is a Reedy fibration (hence objectwise fibration when restricted to
	$\Deltacat$), the
	natural map
	 \[ \stack{X}\times_{\stack{Z}}\stack{Y} \to \stack{X}\ttimes_{\stack{Z}}\stack{Y}
	 \]
 is an objectwise weak equivalence when restricted to $\Deltacat$
 (see Lemma \ref{L:strictvs2-categorical}). It follows
 from \cref{L:levelequiv=>we} that
 the induced map
   	 	 \[ \sing(\stack{X}\times_{\stack{Z}}\stack{Y}) \stackrel{\sim}{\longrightarrow}
   	 	 \sing(\stack{X}\ttimes_{\stack{Z}}\stack{Y})
	 \]
is a weak equivalence of simplicial sets. Precomposing with the isomorphism of
\cref{P:Adjunction}, we obtain the desired weak equivalence
\[ \sing(\stack{X})\times_{\sing(\stack{Z})}\sing(\stack{Y})
    \stackrel{\cong}{\longrightarrow}
\sing(\stack{X}\times_{\stack{Z}}\stack{Y}) \stackrel{\sim}{\longrightarrow}
   	 	 \sing(\stack{X}\ttimes_{\stack{Z}}\stack{Y}).
 \]	
\end{proof}

\section{Lifting lemmas}
\label{S:Lifting}
In this section we prove some lifting lemmas which will be used in the
subsequent sections in the proofs of our main results. We invite the
reader to consult \cref{R:weak} before reading this section to prevent possible
confusion caused by our usage of the term `weak' in what follows.

\subsection{Strictifying lifts}
\label{SS:strictify}

The following lemma is useful when we want to replace a lax solution to  a 
strict lifting problem with a strict solution.

\begin{lemma}
\label{L:strictify}
Consider the following  strictly commutative diagram, where $p$ is a Reedy fibration
of presheaves of groupoids (\cref{D:reedy_cond}) and $i$ is a monomorphism of simplicial sets:

\[
\begin{tikzpicture}[scale=1]

	\node (a) at (0,2) {$\tilde{A}$};
	\node (b) at (0,0) {$\tilde{B}$};
	\node (c) at (2,2) {$\stack{X}$};
	\node (d) at (2,0) {$\stack{Y}$};
	
	\draw[->] (a) to node[left] {$\tilde{i}$} (b);
	\draw[->] (a) to node[above] {$f$} (c);
	\draw[->] (b) to node[below] {$g$} (d);
	\draw[->] (c) to node[right] {$p$} (d);

\end{tikzpicture}
\]	
Suppose that there exists a lift $h$ and 2-isomorphisms $\beta$ and $\gamma$
making the following diagram 2-commutative:		
	\[
\begin{tikzpicture}[scale=1]

	\node (a) at (0,2) {$\tilde{A}$};
	\node (b) at (0,0) {$\tilde{B}$};
	\node (c) at (2,2) {$\stack{X}$};
	\node (d) at (2,0) {$\stack{Y}$};
	
	\draw[->] (a) to node[left] {$\tilde{i}$} (b);
	\draw[->] (a) to node[above] {$f$} (c);
	\draw[->] (b) to node[below] {$g$} (d);
	\draw[->] (c) to node[right] {$p$} (d);
	
	\draw[->] (b) to node[above left,inner sep = 1pt] {$h$} (c);

	\draw[dbl] (1,0) to [bend right=0] node[below right,inner sep=1pt] {$\gamma$} (2,1);
	\draw[dbl] (1,2) to [bend right=0] node[above left,inner sep=1pt] {$\beta$} (0,1);

\end{tikzpicture}
\]
Then, we can replace $h$ by a 2-isomorphic morphism $h'$ so that $\beta$ and $\gamma$ 
become the identity 2-isomorphisms. More precisely, $h'\circ\tilde{i}=f$, $p\circ h'=g$, 
and there is $\theta : h' \Rightarrow h$ such that $\theta\circ\tilde{i}=\beta$ and
$p\circ\theta=\gamma$.
\end{lemma}

\begin{proof} By \cref{P:reedy cond2}, the natural  map
 \begin{align*}
		\Psi : \Hom_{\presheafgpd}(\tilde{B},\stack{X}) \to
		 \Hom_{\presheafgpd}(\tilde{A},\stack{X})
		 \times_{\Hom_{\presheafgpd}(\tilde{A},\stack{Y})}
		  \Hom_{\presheafgpd}(\tilde{B},\stack{Y})
\end{align*}
is a  fibration of groupoids.
The map $h$ can be regarded as an object on the left hand side, with
$\Psi(h)=(h\circ \tilde{i},p\circ  h)$. Since $\Psi$ is a fibration, we can lift the
2-isomorphism $(\beta,\gamma) : (f,g) \Rightarrow (h\circ \tilde{i},p\circ  h)$ to
a 2-isomorphism $\theta : h'\to h$. This is exactly what we need.  
%
%
%
%
\end{proof}

In most of our applications of the above lemma, we will have $B=\Delta^n$, in which case
$\tilde{B}=|\Delta^n|$.

%

\begin{cor}
\label{C:adjust_htpy}
Let $p :\stack{X} \to \stack{Y}$ be a Reedy fibration of presheaves of groupoids, 
$A$ a simplicial set, and $H : \widetilde{A\times\Delta^1} \to \stack{X}$ a 
restricted homotopy (see \cref{SS:restricted}) relative to $\stack{Y}$ starting at 
$H_0:=H|_{\tilde{A}\times\{0\}}: \tilde{A} \to \stack{X}$. Then, for every 
2-isomorphism $\beta : f' \Rightarrow H_0$, there exists a restricted homotopy 
\[H' : \widetilde{A\times\Delta^1} \to \stack{X} \text{ relative to } \stack{Y}\] 
and 
a 2-isomorphism $\Theta : H' \Rightarrow H$ such that $p\circ \Theta =
p\circ\beta\circ\widetilde{\pr_1}$ as 2-isomorphisms
\[p\circ f' \circ\widetilde{\pr_1} 
\Rightarrow p\circ H_0 \circ\widetilde{\pr_1} \,(=p\circ H)  \] 
(i.e., $\Theta$ is relative to $p\circ\beta\circ\widetilde{\pr_1}$), and that
\[f'=H'_0:=H'|_{\tilde{A}\times\{0\}} \text{ and } 
\beta=\Theta_0:=\Theta|_{\tilde{A}\times\{0\}}.\]
\end{cor}

\begin{proof}
 With the notation of \cref{L:strictify},
 let $B=A\times\Delta^1$, $i : A \to A\times\Delta^1$ the inclusion at time $0$, $f=f'$,
 $g=p\circ f' \circ\widetilde{\pr_1}$,  
 $h=H$, $\beta=\beta$ and $\gamma=p\circ\beta\circ\widetilde{\pr_1}$, as in the diagram
	\[
\begin{tikzpicture}[scale=1]

	\node (a) at (0,2) {$\tilde{A}$};
	\node (b) at (0,0) {$\widetilde{A\times\Delta^1}$};
	\node (c) at (3,2) {$\stack{X}$};
	\node (d) at (3,0) {$\stack{Y}$};
	
	\draw[->] (a) to node[left] {$\tilde{i}$} (b);
	\draw[->] (a) to node[above] {$f'$} (c);
	\draw[->] (b) to node[below] {$p\circ f' \circ\widetilde{\pr_1}$} (d);
	\draw[->] (c) to node[right] {$p$} (d);
	
	\draw[->] (b) to node[above left,inner sep = 1pt] {$H$} (c);

	\draw[dbl] (2,0) to [bend right=0] node[below right,inner sep=1pt] {$\gamma$} (3,1);
	\draw[dbl] (1,2) to [bend right=0] node[above left,inner sep=1pt] {$\beta$} (0,1);

\end{tikzpicture}
\]
 The result now follows from	  \cref{L:strictify}.
\end{proof}

\subsection{Strict lifts for Serre+Reedy fibrations}
\label{SS:strictliftingSerre}

From now on, we will assume that our simplicial sets $A$ and $B$ are finite  
non-singular simplicial sets  (see \cref{C:fakification2} and the preceding
paragraph).  For example, $\Delta^n$, $\partial\Delta^n$ 
and $\Lambda^n_k$ have this property.  If $A$ and $B$ have this property, 
then $A\times B$ also has this property, and so does any colimit $A\stackvee_CB$, as 
long as the maps $C \to A$ and $C \to B$ are monomorphisms. In particular, if 
$i: A \to B$ is a monomorphism, then the mapping cylinder $\Cyl(i)$ has this 
property.

\begin{lemma}
\label{L:fake_lift}
 Let $p : \stack{X}\to \stack{Y}$  be a morphism of Serre topological stacks and
  $i : A \to B$   a monomorphism of finite non-singular simplicial sets. If $p$ is a (weak)
  Serre fibration and either $p$ or $i$ is a weak equivalence, then
  $\tilde{i}: \tilde{A} \to \tilde{B}$ has (weak) LLP with respect to $p$ (see \cref{D:WLLP}).
\end{lemma}

\begin{proof}
Consider the lifting problem
	\[
\begin{tikzpicture}[scale=1]
	\node (a) at (0,2) {$\tilde{A}$};
	\node (b) at (0,0) {$\tilde{B}$};
	\node (c) at (2,2) {$\stack{X}$};
	\node (d) at (2,0) {$\stack{Y}$};
	
	\draw[->] (a) to node[left] {$\tilde{i}$} (b);
	\draw[->] (a) to node[above] {$f$} (c);
	\draw[->] (b) to node[below] {$g$} (d);
	\draw[->] (c) to node[right] {$p$} (d);
	
	\draw[dbl] (1,2) to [bend left=30] node[below right] {$\alpha$} (0,1);
	
\end{tikzpicture}
\]
First note that to solve it we are allowed to replace each of $f$ and $g$ with 
a 2-isomorphic morphism (and adjust $\alpha$ accordingly). 
So, we may assume, by \cref{C:fakification2}, that there are maps 
$f' : |A| \to \stack{X}$ and $g' : |B| \to \stack{Y}$ such that 
$f=f'\circ\psi_A$ and $g=g'\circ \psi_B$. Here, $\psi_A : \tilde{A} \to |A|$ 
is as in \cref{Eq:psi}.
Thus, our lifting problem translates to
	\[
\begin{tikzpicture}[scale=1]
	\node (a) at (0,2) {$|A|$};
	\node (b) at (0,0) {$|B|$};
	\node (c) at (2,2) {$\stack{X}$};
	\node (d) at (2,0) {$\stack{Y}$};
	
	\draw[->] (a) to node[left] {$|i|$} (b);
	\draw[->] (a) to node[above] {$f'$} (c);
	\draw[->] (b) to node[below] {$g'$} (d);
	\draw[->] (c) to node[right] {$p$} (d);
	
	\draw[dbl] (1,2) to [bend left=30] node[below right] {$\alpha'$} (0,1);
	
\end{tikzpicture}
\]
(The existence of the unique $\alpha'$ is guaranteed by \cref{C:fakification2}.) 
This problem can now be solved under the given assumptions. Precomposing with the 
$\psi$ maps, we obtain a solution to the original lifting problem. (Also see
 \cref{P:trivserre}.)
\end{proof}

\begin{lemma}
\label{L:strict_lift}
Let $p : \stack{X}\to \stack{Y}$  be a morphism of Serre topological stacks and 
$i : A \to B$   a monomorphism of finite non-singular simplicial sets. If $p$ is a 
Serre fibration and also a Reedy fibration, and either $p$ or $i$ is a weak equivalence, 
then $\tilde{i}: \tilde{A} \to \tilde{B}$ has strict LLP with respect to $p$. That is, 
if  in the diagram
	\[
\begin{tikzpicture}[scale=1]
	\node (a) at (0,2) {$\tilde{A}$};
	\node (b) at (0,0) {$\tilde{B}$};
	\node (c) at (2,2) {$\stack{X}$};
	\node (d) at (2,0) {$\stack{Y}$};
	
	\draw[->] (a) to node[left] {$\tilde{i}$} (b);
	\draw[->] (a) to node[above] {$f$} (c);
	\draw[->] (b) to node[below] {$g$} (d);
	\draw[->] (c) to node[right] {$p$} (d);
	
	\draw[->,dashed] (b) to node[above left] {$h$} (c);
	
\end{tikzpicture}
\]
the outer square is strictly commutative, then there exists a lift $h$ making both 
triangles strictly commutative.	
\end{lemma}

\begin{proof}
  First use \cref{L:fake_lift} to find a solution $h$ which makes the two triangles
  commutative up to 2-isomorphism. Then use \cref{L:strictify} to
  rectify $h$ to make the triangles strictly commutative.
\end{proof}

\begin{cor}
\label{C:strict_lift2}
Assumptions being as in \cref{L:strict_lift}, the map
	 \begin{align*}
		\Hom_{\presheafgpd}(\tilde{B},\stack{X}) \to
		 \Hom_{\presheafgpd}(\tilde{A},\stack{X})
		 \times_{\Hom_{\presheafgpd}(\tilde{A},\stack{Y})}
		  \Hom_{\presheafgpd}(\tilde{B},\stack{Y})
\end{align*}
and, equivalently (see \cref{L:restriction isom}), the map
 	\begin{align*}
		\Hom_{\sgpd}(B,\stack{X}\Del) \to
		 \Hom_{\sgpd}(A,\stack{X}\Del)
		 \times_{\Hom_{\sgpd}(A,\stack{Y}\Del)}
		  \Hom_{\sgpd}(B,\stack{Y}\Del)
	\end{align*}
are fibrations of groupoids that are surjective  on objects (hence, also on morphisms).	
\end{cor}

\begin{proof}
Surjectivity on objects is simply a restatement of \cref{L:strict_lift}.
They are fibrations by \cref{P:reedy cond2}.
\end{proof}

\subsection{Strict lifts for weak Serre+Reedy fibrations}
\label{SS:strictliftingweakSerre}

\begin{lemma}
\label{L:weak_lifting}
Let $p : \stack{X}\to \stack{Y}$  be a morphism of Serre topological stacks and $i : A \to B$ 
a monomorphism of finite non-singular simplicial sets. If $p$ is a weak Serre fibration 
and also a Reedy fibration, and either $p$ or $i$ is a weak equivalence, then 
$\tilde{i}: \tilde{A} \to \tilde{B}$ has \emph{strict} WLLP with respect to $p$ 
in the following sense.  If in the diagram
	\[
\begin{tikzpicture}[scale=1]
	\node (a) at (0,2) {$\tilde{A}$};
	\node (b) at (0,0) {$\tilde{B}$};
	\node (c) at (2,2) {$\stack{X}$};
	\node (d) at (2,0) {$\stack{Y}$};
	
	\draw[->] (a) to node[left] {$\tilde{i}$} (b);
	\draw[->] (a) to node[above] {$f$} (c);
	\draw[->] (b) to node[below] {$g$} (d);
	\draw[->] (c) to node[right] {$p$} (d);
	
	\draw[->,dashed] (b) to node[above left] {$h$} (c);
	
\end{tikzpicture}
\]
 the outer square is strictly commutative, then there exists a lift $h$ and a
 morphism $H : \widetilde{A\times\Delta^1} \to \stack{X}$ such that
 \begin{itemize}
   \item[i)] the lower triangle is strictly commutative, and
   \item[ii)] $H$ is a strict restricted fiberwise homotopy from $f$
   to $h\circ\tilde{i}$ relative to $\stack{Y}$ (see \cref{SS:restricted}), where
   strictness means that  $H_0=f$ and $H_1=h\circ\tilde{i}$.
 \end{itemize}  	
\end{lemma}

\begin{proof}
 First use \cref{L:fake_lift} to find a solution $h$ which makes the lower triangle
  commutative up to a 2-isomorphism and the upper triangle commutative
  up to fiberwise homotopy  $H' : \tilde{A}\times[0,1] \to \stack{X}$,
  as in the diagram
	\[
\begin{tikzpicture}[scale=1]

	\node (a) at (0,2) {$\tilde{A}$};
	\node (b) at (0,0) {$\tilde{B}$};
	\node (c) at (2,2) {$\stack{X}$};
	\node (d) at (2,0) {$\stack{Y}$};
	
	\draw[->] (a) to node[left] {$\tilde{i}$} (b);
	\draw[->] (a) to node[above] {$f$} (c);
	\draw[->] (b) to node[below] {$g$} (d);
	\draw[->] (c) to node[right] {$p$} (d);
	
	\draw[->] (b) to node[above left,inner sep=1pt] {$h$} (c);
	
	\draw[dbl] (1,0) to [bend right=0] node[below right,inner sep=1pt] {$\gamma$} (2,1);
	
	\node at (0.45,1.55) {$H'$};
\end{tikzpicture}
\]
First, we rectify $H'$ using \cref{C:adjust_htpy}, as follows. (Note that \cref{C:adjust_htpy} 
only applies to restricted homotopy and not ordinary homotopy, so we need to replace 
$H'$ by the corresponding restricted homotopy $H : \widetilde{A\times\Delta^1} \to \stack{X}$;
see \cref{L:htpy=>fakehtpy}.) Since $p$ is a Reedy fibration and 
$\widetilde{A\times\Delta^1}$ is cofibrant, \cref{P:reedy cond2}
(with $B=A\times\Delta^1$ and $A$ the empty set)  implies that
\[\Hom_{\presheafgpd}(\widetilde{A\times\Delta^1},\stack{X}) 
\to\Hom_{\presheafgpd}(\widetilde{A\times\Delta^1},\stack{Y})\] 
is a fibration of groupoids. So,
we can replace $H$ by a 2-isomorphic map
so that it becomes relative to 
$\stack{Y}$ (namely,  $p\circ f \circ \widetilde{\pr}_1 = p \circ H$ on the nose); see
\cref{R:restrictedfiberwisehomotopy} to see why this is allowed. 
We can now use  \cref{C:adjust_htpy} to rectify $H$ so that $H_0=f$.

There are two more things to do now: ensure that the 2-isomorphism 
$\epsilon_1 : H_1 \Rightarrow h\circ\tilde{i}$ becomes an equality, and that 
$\gamma=\id$. This is achieved by  applying \cref{L:strictify} to the diagram
	\[
	\begin{tikzpicture}[scale=1]
	
	\node (a) at (0,2) {$\tilde{A}$};
	\node (b) at (0,0) {$\tilde{B}$};
	\node (c) at (2,2) {$\stack{X}$};
	\node (d) at (2,0) {$\stack{Y}$};
	
	\draw[->] (a) to node[left] {$\tilde{i}$} (b);
	\draw[->] (a) to node[above] {$H_1$} (c);
	\draw[->] (b) to node[below] {$g$} (d);
	\draw[->] (c) to node[right] {$p$} (d);
	
	\draw[->] (b) to node[above left,inner sep=1pt] {$h$} (c);
	
	\draw[dbl] (1,0) to [bend right=0] node[below right,inner sep=1pt] {$\gamma$} (2,1);
	\draw[dbl] (1,2) to [bend right=0] node[above left,inner sep=1pt] {$\epsilon_1$} (0,1);
	
\end{tikzpicture}
\]
  to adjust $h$ so that $\epsilon_1$ and $\gamma$ become the identity 2-isomorphisms.
\end{proof}

Let $p : \stack{X}\to \stack{Y}$  be a map of presheaves of groupoids and
$i : A \to B$ a map of simplicial sets. Let $L$  be the groupoid 
\begin{align*}
L:= & \ \Hom_{\presheafgpd}(\widetilde{A\times\Delta^1},\stack{X})_p
		\times_{\Hom_{\presheafgpd}(\tilde{A},\stack{X})}
		\Hom_{\presheafgpd}(\tilde{B},\stack{X}) \\
  \cong  & \ 		
         \Hom_{\sgpd}(A\times\Delta^1, \stack{X}\Del)_p
  \times_{\Hom_{\sgpd}(A, \stack{X}\Del)}
  \Hom_{\sgpd}(B, \stack{X}\Del) 
\end{align*}
of pairs $(H,h)$, where $h : \tilde{B} \to \stack{X}$
is a morphism and $H : \widetilde{A\times\Delta^1} \to  \stack{X}$
is a restricted fiberwise homotopy relative to $\stack{Y}$  such that $H_1=h\circ\tilde{i}$.
Here $H_1:\tilde{A} \to \stack{X}$  stands for the precomposition of
$H$ with the time $1$ inclusion map $\tilde{A} \to \widetilde{A\times\Delta^1}$, 
and the subscript $p$ stands for `fiberwise  relative to $\stack{Y}$'. More precisely,
\begin{align*} 
 \Hom_{\presheafgpd}(\widetilde{A\times\Delta^1},\stack{X})_p:= & \
  \Hom_{\presheafgpd}(\widetilde{A\times\Delta^1},\stack{X})
  \times_{\Hom_{\presheafgpd}(\widetilde{A\times\Delta^1},\stack{Y})}  
  \Hom_{\presheafgpd}(\tilde{A},\stack{Y}) \\
  \cong & \ 
   \Hom_{\sgpd}(A\times\Delta^1,\stack{X}\Del)
  \times_{\Hom_{\sgpd}(A\times\Delta^1,\stack{Y}\Del)}  
  \Hom_{\sgpd}(A,\stack{Y}\Del),
 \end{align*}
where the first map in the  fiber product is induced by $p$, and the 
second map is induced  by the  projection $\widetilde{A\times\Delta^1} \to \tilde{A}$.

Thus, we have isomorphisms of groupoids
\begin{align*} 
L\cong   & \
  \Hom_{\presheafgpd}(\widetilde{A\times\Delta^1},\stack{X})
   \times_{\Hom_{\presheafgpd}(\widetilde{A\times\Delta^1},\stack{Y})}
   \Hom_{\presheafgpd}(\tilde{A},\stack{Y}) \\
   &  \ \ \ \ \ \ \ \ \ \ \ \ \times_{\Hom_{\presheafgpd}(\tilde{A},\stack{X})}
		\Hom_{\presheafgpd}(\tilde{B},\stack{X}) \\ 
	\cong	&  
 \Hom_{\sgpd}(A\times\Delta^1,\stack{X}\Del)
    \times_{\Hom_{\sgpd}(A\times\Delta^1,\stack{Y}\Del)}  
  \Hom_{\sgpd}(A,\stack{Y}\Del) \\
      & \ \ \ \ \ \ \ \ \ \ \ \  \times_{\Hom_{\sgpd}(A, \stack{X}\Del)}
    \Hom_{\sgpd}(B, \stack{X}\Del). 	
\end{align*} 

%
%
%

\begin{cor}
\label{C:weak_lifting2}
Notation being as above and assumptions being as in \cref{L:weak_lifting},
the map
\begin{align*}
		L  & \to
		 \Hom_{\presheafgpd}(\tilde{A},\stack{X})
		 \times_{\Hom_{\presheafgpd}(\tilde{A},\stack{Y})}
		  \Hom_{\presheafgpd}(\tilde{B},\stack{Y}) \\
		(H,h) & \mapsto (H_0,p\circ h)
\end{align*}
and, equivalently (see \cref{L:restriction isom}), the map
\begin{align*}
		L  & \to
		 \Hom_{\sgpd}(A,\stack{X}\Del)
		 \times_{\Hom_{\sgpd}(A,\stack{Y}\Del)}
		  \Hom_{\sgpd}(B,\stack{Y}\Del)
\end{align*}
are fibrations of groupoids that are surjective  on objects (hence, also on morphisms,
as well as tuples of composable morphisms).
\end{cor}

\begin{proof} 
Let us denote the map in question by $\Psi$. 
The surjectivity of $\Psi$
on objects is simply a restatement of \cref{L:weak_lifting}. Let us spell this out.
Consider an object in
 \[		 \Hom_{\presheafgpd}(\tilde{A},\stack{X})
		 \times_{\Hom_{\presheafgpd}(\tilde{A},\stack{Y})}
		  \Hom_{\presheafgpd}(\tilde{B},\stack{Y}), \]
namely a pair $(f,g)$ making the outer square in the following diagram strictly commutative:
	\[
\begin{tikzpicture}[scale=1]
	\node (a) at (0,2) {$\tilde{A}$};
	\node (b) at (0,0) {$\tilde{B}$};
	\node (c) at (2,2) {$\stack{X}$};
	\node (d) at (2,0) {$\stack{Y}$};
	
	\draw[->] (a) to node[left] {$\tilde{i}$} (b);
	\draw[->] (a) to node[above] {$f$} (c);
	\draw[->] (b) to node[below] {$g$} (d);
	\draw[->] (c) to node[right] {$p$} (d);
	
	\draw[->,dashed] (b) to node[above left] {$h$} (c);
	
\end{tikzpicture}
\]
By \cref{L:weak_lifting}, this lifting problem has a weak solution $(H,h)$,
namely $h : \tilde{B} \to \stack{X}$ and 
$H : \widetilde{A\times\Delta^1} \to \stack{X}$ such that
 \begin{itemize}
   \item[i)] the lower triangle is strictly commutative, and
   \item[ii)] $H$ is a strict restricted fiberwise homotopy from $f$
   to $h\circ\tilde{i}$ relative to $\stack{Y}$, where
   strictness means that  $H_0=f$ and $H_1=h\circ\tilde{i}$.
 \end{itemize}  	
By definition of $L$, 
such a pair determines an object in
$L$ mapping to the pair $(f,g)$, $\Psi(H,h)=(f,g)$. 
This proves surjectivity on objects.

To prove fibrancy, suppose in the above setting that we are also given 
2-isomorphisms $\beta : f' \Rightarrow f$ and $\gamma :g' \Rightarrow g$
such that $p\circ \beta  = \gamma\circ \tilde{i}$. 
We need to construct a pair $(\Theta, \theta) \in \mor(L)$ with the following
properties:
    \begin{itemize}
   \item[i)] $\theta:  h' \Rightarrow h$ is relative to 
   $\gamma$ (that is, $p\circ\theta =\gamma$),
   \item[ii)] $\Theta : H' \Rightarrow H$ is relative to 
   $\gamma\circ \tilde{i}\circ\widetilde{\pr_1}=p\circ \beta \circ\widetilde{\pr_1}$
   (that is, $p\circ \Theta = \gamma \circ \tilde{i} \circ \widetilde{\pr_1}$), 
    $\Theta_0=\beta$ and $\Theta_1=\theta\circ \tilde{i}$.
 \end{itemize}

By \cref{C:adjust_htpy},
we have a restricted fiberwise homotopy $H' : \widetilde{A\times\Delta^1} \to \stack{X}$ 
relative to $\stack{Y}$,
and a 2-isomorphism $\Theta : H' \Rightarrow H$ relative
to $p\circ\beta\circ\widetilde{\pr_1}$ such that
$f'=H'_0$ and $\beta=\Theta_0$. This is our desired $\Theta$. 

To find $\theta$, note that
its restriction to $\tilde{A}$ is already determined, namely $\Theta_1$. So, we need
to extend $\Theta_1$ to the whole of $\tilde{B}$ in such a way that
$p\circ\theta=\gamma$. We do this by solving the following
lifting problem for $(h', \theta)$:
	\[
\begin{tikzpicture}[scale=1]
	\node (a) at (0,3) {$\tilde{A}$};
	\node (b) at (0,0) {$\tilde{B}$};
	\node (c) at (4,3) {$\stack{X}$};
	\node (d) at (4,0) {$\stack{Y}$};
	
	\draw[->] (a) to node[left] {$\tilde{i}$} (b);
	\draw[->] (a) to [bend left = 15]   node[below, inner sep=1pt] (H') {}
	      node[above, pos=0.4] {$H_1'$}  (c);
	\draw[->] (a) to [bend right = 15]  node[above, inner sep=1pt] (H) {}
	      node[below, pos=0.4] {$H_1$} (c);
	\draw[->] (b) to [bend left = 15]   node[below, inner sep=1pt] (g') {}
	      node[above] {$g'$} (d);
	\draw[->] (b) to [bend right = 15]  node[below, inner sep=1pt] (g) {}
	      node[below] {$g$} (d);
	\draw[->] (c) to node[right] {$p$} (d);
	
	\draw[->,dashed] (b) to  [bend left = 15]   node[below, inner sep=1pt] (h') {}
	      node[above left, pos=0.3] {$h'$} (c);
	\draw[->]        (b) to  [bend right = 15]  node[below, inner sep=1pt] (h) {}
	      node[below right, pos=0.7] {$h$} (c);
	
	\draw[double,
              double equal sign distance,
              -implies,
              shorten >= 2pt,
              shorten <= 2pt
              ]
              (H') to node[left] {$\Theta_1$} (H);	
     \draw[double,
              double equal sign distance,
              -implies,
              shorten >= 2pt,
              shorten <= 2pt
              ]
              (g') to node[left] {$\gamma$} (g);	 
    \draw[double, dashed,
              double equal sign distance,
              -implies,
              shorten >= 2pt,
              shorten <= 2pt
              ]
              (h') to node[left] {$\theta$} (h);	                   
	
\end{tikzpicture}
\]
Existence of a solution is guaranteed by 
\cref{P:reedy cond2}. 
\end{proof}

%
%
%
%

\section{Singular functor preserves fibrations}
\label{S:Kan}
In this section we study the effect of the  functor $\sing$ on fibrations of stacks.
We begin with a simple example to show why the Reedy condition is necessary in the
statement of our main result (\cref{T:sing of WSR fib}).

\begin{ex}
Let $X$ be a trivial groupoid with more than one point, namely one that is 
equivalent but not equal to a point. 
Let $\stack{X}$ be the constant presheaf with value $X$ (viewed as a stack).
Pick a point in $X$ and consider the map $*\to \stack{X}$. This map
is an equivalence of stacks, hence is a Serre fibration. However, the
induced map of simplicial sets
 \[\sing(*)=* \to N(X)=\sing(\stack{X})\]
is  not a Kan fibration.
\end{ex}

\subsection{Weak Kan fibrations}
\label{SS:weakKan}

In what follows, the homotopy groups $\pi_n(X,x)$ of a simplicial set $X$ which is not 
necessarily Kan are taken to be those of its geometric realization.

\begin{defn}
\label{D:weak kan}
	We say that a map of simplicial sets $p:X \to Y$ is 
    a \emph{weak  Kan fibration} if for any
    trivial cofibration  $i : A \to \Delta^n$, 
    every lifting problem
\[
\begin{tikzpicture}[scale=1]

	\node (a) at (0,2) {$A$};
	\node (b) at (0,0) {$\Delta^n$};
	\node (c) at (2,2) {$X$};
	\node (d) at (2,0) {$Y$};
	
	\draw[->] (a) to node[left] {$i$} (b);
	\draw[->] (a) to node[above] {$f$} (c);
	\draw[->] (b) to node[below] {$g$} (d);
	\draw[->] (c) to node[right] {$p$} (d);
	
	\draw[->,dashed] (b) to node[above left] {$h$} (c);
	
\end{tikzpicture}
\]
	has a weak solution; namely, 
	there exists $h:\Delta^n \to X$  such that the bottom triangle commutes and
	$f: A \to X$ is fiberwise homotopic to
	$h \circ i : A \to X$ relative to $Y$.
	We say that $p$ is 
    a \emph{weak trivial Kan fibration}  if it is a weak  Kan fibration
    and, in addition, it has the weak lifting property with respect to
    the inclusions $\partial\Delta^n \to \Delta^n$, $n\geq 0$. 
\end{defn}


In the above definition, a fiberwise homotopy relative to $Y$ means
a map of simplicial sets $H :  A\times\Delta^1 \to X$ such that $p\circ H$ is
the trivial homotopy from $p\circ f$ to itself.

\begin{rem}
We do not know if the above definition is the ``correct'' simplicial counterpart
of the notion of a weak Serre fibration, but it serves our purposes in this paper
(thanks to Lemma \ref{L:d^*}).
It is not clear to us whether a weak Kan fibration will have the weak left lifting 
property with respect to \emph{all} trivial cofibrations.
\end{rem}

\begin{lemma}
\label{L:trivialgamma}
Let $p: X \to Y$ be a trivial weak Kan fibration. Assume that $Y$ is a Kan 
simplicial set,
and that there exists a 
Kan simplicial set $X'$  together with a weak equivalence $X' \to X$. 
Then, $p$ is a weak equivalence.
\end{lemma}

\begin{proof} 
 First we prove that $\pi_n(p)$ is injective. Let $x$ be a base point 
 that is in the image of $X'$, and let $y=p(x)$. The fact that  $X'$ is Kan guarantees
 that any class in $\pi_n(X,x)$ is represented by a pointed map
 $f : \partial\Delta^n \to X$. If the image of this class in $\pi_n(Y,y)$ is trivial, we will
 have, since $Y$ is Kan,  a filling  $g$ for $p\circ f$, as in the diagram
    \[
\begin{tikzpicture}[scale=1]
	\node (a) at (0,2) {$\partial\Delta^n$};
	\node (b) at (0,0) {$\Delta^n$};
	\node (c) at (2,2) {$X$};
	\node (d) at (2,0) {$Y$};
	
	\draw[->] (a) to node[left] {$i$} (b);
	\draw[->] (a) to node[above] {$f$} (c);
	\draw[->] (b) to node[below] {$g$} (d);
	\draw[->] (c) to node[right] {$p$} (d);
	
	\draw[->,dashed] (b) to node[above left] {$h$} (c);
\end{tikzpicture}
\]
So, a lift $h$ exists which makes the diagram commutative (possibly after
replacing $f$ by a fiberwise homotopic map). This implies that the class represented
by $f$ in $\pi_n(X,x)$  is trivial.

To prove surjectivity of $\pi_n(p)$, let $g : \partial\Delta^n \to Y$ represent an
arbitrary class in  $\pi_n(Y,y)$. To lift this to $X$, we begin
by lifting 
$g|_{\Lambda_0^n} : \Lambda_0^n \to Y$ to $X$. To do so, first extend 
$g|_{\Lambda_0^n}$ to the whole $\Delta^n$ using the Kan property of $Y$.
Then, apply the weak lifting property to the
trivial cofibration $\{0\} \to \Delta^n$. Restricting the outcome to
$\Lambda_0^n$, we find a lift
$\hat{g} : \Lambda_0^n \to X$, sending $0$ to a point that is fiberwise homotopic to $x$.
Since $\pi_n(X,x)$ and $\pi_n(X,\hat{g}(0))$ have the same image in 
$\pi_n(Y,y)$, as can be seen by 
passing to the geometric realization, there is no harm in replacing $x$ with $\hat{g}(0)$.
So we may assume that $\hat{g}(0)=x$. Now consider the following lifting problem
      \[
\begin{tikzpicture}[scale=1]
	\node (a) at (0,2) {$\partial\Delta^{n-1}$};
	\node (b) at (0,0) {$\Delta^{n-1}$};
	\node (c) at (3,2) {$X$};
	\node (d) at (3,0) {$Y$};
	
	\draw[->] (a) to node[left] {$j$} (b);
	\draw[->] (a) to node[above] {$\hat{g}\circ (d_0|_{\partial\Delta^{n-1}})$} (c);
	\draw[->] (b) to node[below] {$g \circ d_0$} (d);
	\draw[->] (c) to node[right] {$p$} (d);
	
	\draw[->,dashed] (b) to node[above left] {$h$} (c);
\end{tikzpicture}
\]
Here, $d_0 : \Delta^{n-1} \to \partial\Delta^{n}$ is the $0^{th}$ face of
$\partial\Delta^{n}$ and $j : \partial\Delta^{n-1} \to \Delta^{n-1}$ is the inclusion map.
A weak solution to this problem can be glued to $\hat{g}$ to give a map
\[ G : \Lambda_0^n \stackvee_{\partial\Delta^{n-1}}  (\partial\Delta^{n-1}\times \Delta^1) 
\stackvee_{\partial\Delta^{n-1}} \Delta^{n-1}  \longrightarrow X\]
making the following diagram commutative
        \[
\begin{tikzpicture}[scale=1]
	\node (a) at (0,2) 
	   {$\Lambda_0^n \stackvee_{\partial\Delta^{n-1}}  (\partial\Delta^{n-1}\times \Delta^1)
       \stackvee_{\partial\Delta^{n-1}} \Delta^{n-1} $};
	\node (b) at (0,0) {$\partial\Delta^{n}$};
	\node (c) at (5,2) {$X$};
	\node (d) at (5,0) {$Y$};
	
	\draw[->] (a) to node[left] {$P$} (b);
	\draw[->] (a) to node[above] {$G$} (c);
	\draw[->] (b) to node[below] {$g$} (d);
	\draw[->] (c) to node[right] {$p$} (d);
	
\end{tikzpicture}
\]
Here, 
 \[P: \Lambda_0^n \stackvee_{\partial\Delta^{n-1}}  (\partial\Delta^{n-1}\times \Delta^1)
       \stackvee_{\partial\Delta^{n-1}} \Delta^{n-1} \to \partial\Delta^{n}\] 
is the map that collapses $\partial\Delta^{n-1}\times \Delta^1$
to $\partial\Delta^{n-1}$ via  the first projection; note that 
\[|\Lambda_0^n \stackvee_{\partial\Delta^{n-1}}  (\partial\Delta^{n-1}\times \Delta^1)
       \stackvee_{\partial\Delta^{n-1}} \Delta^{n-1}|\]
        is homeomorphic to an $n$-sphere. 
The (geometric realization of the) map $G$ represents a lift of the class in $\pi_n(Y,y)$
represented by $g$  to a class in class in $\pi_n(X,x)$.
This completes the proof of surjectivity.
\end{proof}

\begin{lemma}
\label{L:ob and mor kan}
 Let $p : \stack{X}\to \stack{Y}$  be a morphism of Serre topological stacks
 that is a (weak) (trivial) Serre fibration and a Reedy fibration. 
 Let $R_0(\stack{X})=\ob(\stack{X}\Del)$,
  $R_1(\stack{X})=\mor(\stack{X}\Del)$ and
 \[R_m(\stack{X})=
   R_1(\stack{X})\times_{R_0(\stack{X})}\times\cdots\times_{R_0(\stack{X})}R_1(\stack{X}).
   \]  
 Then, for every $m\geq 0$, the induced map
 \[R_m(\stack{X}) \to R_m(\stack{Y})\]
is a (weak) (trivial) Kan fibration of simplicial sets.
\end{lemma}

\begin{proof}
First, we prove the statement in the case of  a Serre fibration.
Let
$A=\Lambda^n_k$ and $B=\Delta^n$, and 
let  $i : A \to B$ be  the horn inclusion.
By \cref{C:strict_lift2}, we have a fibration of groupoids
	\begin{align*}
		\Hom_{\sgpd}(B,\stack{X}\Del) \to
		 \Hom_{\sgpd}(A,\stack{X}\Del)
		 \times_{\Hom_{\sgpd}(A,\stack{Y}\Del)}
		  \Hom_{\sgpd}(B,\stack{Y}\Del)
	\end{align*}
which is surjective on objects. Taking nerves on both sides,
we find a fibration of simplicial sets 
   	\begin{align*}
		N\Hom_{\sgpd}(B,\stack{X}\Del) \to
		 N\Hom_{\sgpd}(A,\stack{X}\Del)
		 \times_{N\Hom_{\sgpd}(A,\stack{Y}\Del)}
		  N\Hom_{\sgpd}(B,\stack{Y}\Del)
	\end{align*}
which is  surjective on $m$-simplices, for all $m$.	
The surjectivity on $m$-simplices precisely
translates to the fact that $i$ has LLP with respect to $R_m(\stack{X}) \to R_m(\stack{Y})$,
as the above map on the level on $m$-simplices is, term by term, equal to the map
 	\begin{align*}
		\Hom_{\sset}(B,R_m(\stack{X})) \to
		 \Hom_{\sset}(A,R_m(\stack{X}))
		 \times_{\Hom_{\sset}(A,R_m(\stack{Y}))}
		  \Hom_{\sset}(B,R_m(\stack{Y})).
	\end{align*}
This shows that $R_m(\stack{X}) \to R_m(\stack{Y})$
is a Kan fibration.
The case of a trivial Serre fibration is proved similarly (taking $A=\partial\Delta^n$
instead of $\Lambda^n_k$).

Now consider the case where $p$ is a weak  Serre fibration. 
Let $B=\Delta^n$ and  $i : A \to B$ be  
as in  \cref{D:weak kan}. By
\cref{C:weak_lifting2}, we have a fibration of simplicial sets 
\begin{align}\label{Eq:surjective}
		NL  & \to
		 N\Hom_{\sgpd}(A,\stack{X}\Del)
		 \times_{N\Hom_{\sgpd}(A,\stack{Y}\Del)}
		  N\Hom_{\sgpd}(B,\stack{Y}\Del)
\end{align} 
which is  surjective on $m$-simplices, for all $m$.	
By the discussion   just before
Lemma \ref{C:weak_lifting2}, and the fact that taking nerves
commutes with fiber products, $NL$ is isomorphic
to 
\begin{align*} 
& N\Hom_{\sgpd}(A\times\Delta^1,\stack{X}\Del)
    \times_{N\Hom_{\sgpd}(A\times\Delta^1,\stack{Y}\Del)}  
  N\Hom_{\sgpd}(A,\stack{Y}\Del) \\
      &   \ \ \ \ \ \ \ \ \ \ \ \ \ \ \ \ \ \ \ \ \ \ \ \ 
      \times_{N\Hom_{\sgpd}(A, \stack{X}\Del)}
    N\Hom_{\sgpd}(B, \stack{X}\Del). 
\end{align*}
Its set of $m$-simplices  is then equal to
\begin{align*} 
(NL)_m = &
 \Hom_{\sset}(A\times\Delta^1,R_m(\stack{X}))
    \times_{\Hom_{\set}(A\times\Delta^1,R_m(\stack{Y}))}  
  \Hom_{\sset}(A,R_m(\stack{Y})) \\
     &   \ \ \ \ \ \ \ \ \ \ \ \  \ \ \ \ \ \ \ \ \ \ \ \ 
      \times_{\Hom_{\sset}(A, R_m(\stack{X}))}
    \Hom_{\sset}(B, R_m(\stack{X}))   \\
    \cong &
     \Hom_{\sset}(A\times\Delta^1,R_m(\stack{X}))_p  \times_{\Hom_{\sset}(A, R_m(\stack{X}))}
    \Hom_{\sset}(B, R_m(\stack{X})),
\end{align*} 
where 
\[
\Hom_{\sset}(A\times\Delta^1,R_m(\stack{X}))_p:=
\Hom_{\sset}(A\times\Delta^1,R_m(\stack{X}))
    \times_{\Hom_{\set}(A\times\Delta^1,R_m(\stack{Y}))}  
  \Hom_{\sset}(A,R_m(\stack{Y}))
\]
is the set of fiberwise homotopies. Thus, we can think of $(NL)_m$ as the set 
of pairs $(H,h)$, where $h : B \to R_m(\stack{X})$
is a map of simplicial sets and $H : A\times\Delta^1 \to  R_m(\stack{X})$
is a  fiberwise homotopy relative to $R_m(\stack{Y})$  
such that $H_1=h\circ i$.

Hence, on the level of $m$-simplices, the map \ref{Eq:surjective} above
can be identified with the natural map
  	\begin{align*}
		\Hom_{\sset}(A\times\Delta^1,R_m(\stack{X}))_p  
		 \times_{\Hom_{\sset}(A, R_m(\stack{X}))}
          \Hom_{\sset}(B, R_m(\stack{X})) \\
		\ \ \ \ \ \ \ \ \to
		 \Hom_{\sset}(A,R_m(\stack{X}))
		 \times_{\Hom_{\sset}(A,R_m(\stack{Y}))}
		  \Hom_{\sset}(B,R_m(\stack{Y})).
	\end{align*}
which assigns to any weak solution $(H,h)$, viewed as an element in the left hand side,
its associated 	 lifting problem $(f,g)$, viewed as an element in the  right hand side, 
as in the diagram  
	\[
\begin{tikzpicture}[scale=1]

	\node (a) at (0,2) {$A$};
	\node (b) at (0,0) {$B$};
	\node (c) at (2,2) {$R_m(\stack{X})$};
	\node (d) at (2,0) {$R_m(\stack{Y})$};

	\draw[->] (a) to node[left] {$i$} (b);
	\draw[->] (a) to node[above] {$f$} (c);
	\draw[->] (b) to node[below] {$g$} (d);
	\draw[->] (c) to node[right] {$p$} (d);
	
	\draw[dashed,->] (b) to node[above left,inner sep=1pt] {$h$} (c);
	
	\node at (0.45,1.55) {$H$};
\end{tikzpicture}
\]
The surjectivity of this
map precisely means that any such lifting problem has a weak solution.

The case of a weak trivial Serre fibration is proved similarly.
\end{proof}

\subsection{A lemma on $d^*: \sset \to \bsset$}
\label{SS:d*}
In this section  we prove a lemma which is used in the proof of \cref{L:jardine_lemma},
which in turn plays an important role in the proof of our first main result, 
\cref{T:sing of WSR fib}.

First, we briefly recall notion of exterior product of simplicial sets.
Given simplicial sets $X$ and $Y$, their  \emph{exterior product}  is
the bisimplicial set $X\boxtimes Y$ defined by
  \[
   (X\boxtimes Y)_{m,n}:=X_m\times Y_n.
 \]
We have $\Diag(X\boxtimes Y)=X\times Y$. The exterior product has the property that
 the functor
 \begin{align*}
  \sset & \to \bsset, \\
  A & \mapsto A\boxtimes \Delta^n,
 \end{align*}
 is left adjoint to
 \begin{align*}
    \bsset & \to \sset, \\
    X & \mapsto X_{*,n}.
 \end{align*}

Let $A \to B$ be  a map of simplicial sets. Recall  the  
functor $d^*: \sset \to \bsset$ from \cref{D:diag_adjoint}. We have a natural map 
    \[d^*(A) \to A\boxtimes B,\]
namely, the  adjoint (see \cref{P:Adjunction}) to the diagonal inclusion
  \[ A \to \Diag(A\boxtimes B)=A\times B.
  \]
In the next lemma we show that for any monomorphism $A \to \Delta^n$, the
map 
$d^*(A) \to A\boxtimes \Delta^n$ is a trivial cofibration. 
The case $A=\Lambda^n_k$ of the following lemma is proved in \cite{Jardine}
(see \cite{Jardine}, top of the page 221, just before Lemma  3.12).
 
\begin{lemma}
\label{L:d^*} 
Let $\gamma : A \to \Delta^n$ be a cofibration (not necessarily trivial)
of simplicial sets.
Then,  for every $m$, we have
\[(d^*A)_{m,*}=\coprod_{\alpha \in A_m} C_{\alpha}, \]
where $C_{\alpha}\subseteq A$ is the union of all
faces of $A$ that contain $\alpha$.
The natural map  of bisimplicial sets
    \[i : d^*(A) \to A\boxtimes\Delta^{n},\]
namely, the left adjoint to the diagonal inclusion
  \[ (\id,\gamma) : A \to \Diag(A\boxtimes\Delta^n)=A\times\Delta^n,
  \]
is given on the $m^{th}$ column by the inclusion
 \[i_m : \coprod_{\alpha \in A_m} C_{\alpha} \hookrightarrow 
 \coprod_{\alpha \in A_m} \Delta^n.\]
In particular, $i_m$ is a trivial cofibration of simplicial sets for every
$m$ (thus, $i$ is a vertical pointwise trivial cofibration of bisimplicial sets).
\end{lemma}

In the above lemma, by a {\em face} of $A$ we mean the sub simplicial set generated by a
 (non-degenerate) simplex in $A$. Note that such a face is isomorphic to some
 simplex $\Delta^m$ and that $A \subseteq \Delta^n$ is necessarily a union of
 a collection of faces in $\Delta^n$.

\begin{proof}
 First we consider the case where $\gamma$ is the inclusion of a face
 (so $A=\Delta^d$ for some $d$). In this case,
  $d^*(A)=A\boxtimes A$, which can be identified with a 
  sub bisimplicial set of 
 $A\boxtimes\Delta^n$ via the map whose effect on the $m^{th}$ column 
 is given by
     \[i^A_m : \coprod_{\alpha \in A_m} A \xhookrightarrow{\gamma} 
 \coprod_{\alpha \in A_m} \Delta^n \ (\subseteq \coprod_{\alpha \in \Delta^n_m} \Delta^n ).\]
 The key observation here is that, for any two faces $F_{i}$ and $F_{j}$ of $A$,
 the
 image of $i_m^{F_{i}\cap F_{j}}$ in the $m^{\text{th}}$ column 
 $\coprod_{\alpha \in \Delta^n_m} \Delta^n$ of  $\Delta^n\boxtimes\Delta^n$
 is equal to the intersection of the images of $i_m^{F_j}$ and $i_m^{F_k}$.
 
 Now, for general $A$, write is as a coequalizer of the inclusions of its faces, namely
  \[A =\operatorname{coeq}\left(\coprod_{j,k}F_j\cap F_k 
     \rightrightarrows \coprod_{j}F_{j}\right).\]
 Since $d^*$ commutes with colimits, we have    
  \[d^*(A) =\operatorname{coeq}\left(\coprod_{j,k}d^*(F_j\cap F_k) 
     \rightrightarrows \coprod_{j}d^*(F_{j})\right).\]
 The  observation above that  $i_m$
 respects intersections implies that 
 $i^A_m : d^*(A) \to \Delta^n\boxtimes\Delta^n$ is injective and 
 the image under $i^A_m$ of $d^*(A)$ in the $m^{\text{th}}$ column 
 $\coprod_{\alpha \in \Delta^n_m} \Delta^n$ of  $\Delta^n\boxtimes\Delta^n$ is
 the union of images of all  $i^{F_j}_m$. This is
 precisely $\coprod_{\alpha \in A_m} C_{\alpha}$.
\end{proof}

\subsection{A criterion for diagonal fibrations}
\label{SS:diagonal}

To prove our first main result we need a generalization of  Lemma 4.8  
of (\cite{Jardine}, Chapter IV) which we now prove.

%
%
%

In the next lemma,
we are regarding $X$ as the simplicial object $[m] \mapsto X_{m,*}$ in $\sset$.

\begin{lemma}
\label{L:jardine_lemma}
Let  $f : X \to Y$ be a Reedy fibration of bisimplicial sets.
Let $\gamma : A \to \Delta^n$ be a monomorphism.
Suppose  that $\gamma$ has (weak) left lifting property 
with respect to $f_{*,n} : X_{*,n} \to Y_{*,n}$, for all $n$ (\cref{D:weak kan}). 
Then, $\gamma$ has (weak) left lifting property  with respect to 
$\Diag(f) : \Diag(X) \to \Diag(Y)$. In particular,
if each $f_{*,n} : X_{*,n} \to Y_{*,n}$ is a (weak) (trivial) Kan fibration, 
then so is $\Diag(f)$.
\end{lemma}

\begin{proof}
We want to show that  $\gamma : A \to \Delta^n$
has (W)LLP with respect to $\Diag(f)$. By adjunction, this is equivalent to showing that
the lifting problem
\[
\begin{tikzpicture}[scale=1,baseline=(current  bounding  box.center)]

	\node (a) at (0,2) {$d^*(A)$};
	\node (b) at (0,0) {$d^*(\Delta^n)$};
	\node (c) at (2,2) {$X$};
	\node (d) at (2,0) {$Y$};
	
	\draw[->] (a) to node[left] {$d^*(\gamma)$} (b);
	\draw[->] (a) to node[above] {$u$} (c);
	\draw[->] (b) to node[below] {$v$} (d);
	\draw[->] (c) to node[right] {$f$} (d);
	
	\draw[->,dashed] (b) to node[above left] {$h$} (c);
	
\end{tikzpicture}\tag{$\ast$}
\]
in bisimplicial sets has a solution, with the caveat that, in the `weak' setting,
instead of a fiberwise homotopy in the upper triangle of ($\ast$) we should be
asking for a map $d^*(A\times\Delta^1) \to X$ (with the obvious properties).

We solve ($\ast$) in two steps, by writing the left vertical map
\[d^*(A) \to d^*(\Delta^n)=\Delta^{n,n}=\Delta^n\boxtimes \Delta^n
\]
as composition of two inclusions
\[
  d^*(A) \xrightarrow{i}  A\boxtimes \Delta^n
  \xrightarrow{j} \Delta^n\boxtimes \Delta^n.
\]
Here, the map $i$ is  adjoint to the diagonal inclusion
  \[ (\id,\gamma) :A \to \Diag(A\boxtimes \Delta^n)=A\times \Delta^n;
  \]
see the paragraph before \cref{L:d^*}.

\medskip
\noindent{\bf Step 1.}
We first solve the lifting problem
\[
\begin{tikzpicture}[scale=1,baseline=(current  bounding  box.center)]

	\node (a) at (0,2) {$d^*(A)$};
	\node (b) at (0,0) {$A\boxtimes \Delta^n$};
	\node (c) at (2,2) {$X$};
	\node (d) at (2,0) {$Y$};
	
	\draw[->] (a) to node[left] {$i$} (b);
	\draw[->] (a) to node[above] {$u$} (c);
	\draw[->] (b) to node[below] {$v\circ j$} (d);
	\draw[->] (c) to node[right] {$f$} (d);
	
	\draw[->,dashed] (b) to node[above left] {$h$} (c);
	
\end{tikzpicture}\tag{$\ast\ast$}
\]
By \cref{L:d^*}, $i$ is a pointwise  
trivial cofibration, so it has strict LLP with respect
to $f$, as $f$ is a Reedy fibration (see \cite{Jardine}, Chapter IV, Lemma  3.3(1)).
Therefore, our lifting problem has indeed a strict solution.

\medskip
\noindent{\bf Step 2.} We now solve the lifting problem
\[
\begin{tikzpicture}[scale=1,baseline=(current  bounding  box.center)]

	\node (a) at (0,2) {$A\boxtimes \Delta^n$};
	\node (b) at (0,0) {$\Delta^n\boxtimes \Delta^n$};
	\node (c) at (2,2) {$X$};
	\node (d) at (2,0) {$Y$};
	
	\draw[->] (a) to node[left] {$j$} (b);
	\draw[->] (a) to node[above] {$h$} (c);
	\draw[->] (b) to node[below] {$v$} (d);
	\draw[->] (c) to node[right] {$f$} (d);
	
	\draw[->,dashed] (b) to node[above left] {$l'$} (c);
	
\end{tikzpicture}\tag{$\ast\ast\ast$}
\]
Consider the adjoint lifting problem
\[
\begin{tikzpicture}[scale=1,baseline=(current  bounding  box.center)]

	\node (a) at (0,2) {$A$};
	\node (b) at (0,0) {$\Delta^n$};
	\node (c) at (2,2) {$X_{*,n}$};
	\node (d) at (2,0) {$Y_{*,n}$};
	
	\draw[->] (a) to node[left] {$\gamma$} (b);
	\draw[->] (a) to node[above] {} (c);
	\draw[->] (b) to node[below] {} (d);
	\draw[->] (c) to node[right] {$f_{*,n}$} (d);
	
	\draw[->,dashed] (b) to node[above left] {$l$} (c);
	
\end{tikzpicture}
\]

If $\gamma$ has strict LLP with respect to
$f_{*,n} : X_{*,n} \to Y_{*,n}$, this problem has a strict solution.
Hence, our original problem  ($\ast$) also has  a strict solution, and we are done.

If $\gamma$ has weak LLP with respect to
$f_{*,n} : X_{*,n} \to Y_{*,n}$, a lift
$l : \Delta^n \to X_{*,n}$ exists, but the upper triangle commutes only
up to a fiberwise homotopy $H : A\times\Delta^1 \to X_{*,n}$
(relative to $Y_{*,n}$). By adjunction, this gives rise to a lift
\[l' : \Delta^n\boxtimes \Delta^n  \to X
\]
in ($\ast\ast\ast$). The upper triangle in ($\ast\ast\ast$), however, is not,
strictly speaking, homotopy commutative. Rather, instead of a homotopy we have a map
$H' : (A\times\Delta^1)\boxtimes \Delta^n \to X$, the adjoint of $H$.
Let $H''$ be the composition
 \[H'' : d^*(A\times\Delta^1) \to (A\times\Delta^1)\boxtimes \Delta^n
  \xrightarrow{H'}  X.
 \]
Here, the first map is  adjoint to
\[(\id,\gamma) \times \id_{\Delta^1} :
  A \times\Delta^1 \to \Diag((A\times\Delta^1)\boxtimes \Delta^n)=
(A\times\Delta^1)\times \Delta^n=A\times \Delta^n\times\Delta^1,
\]
where $(\id,\gamma) : A \to A\times \Delta^n$ is the diagonal inclusion;
see the paragraph before \cref{L:d^*}.
It follows that the pair
  \[l' : \Delta^n \boxtimes \Delta^n  \to X, \ \ H'' : d^*(A\times\Delta^1) \to X
\]
is the desired solution to ($\ast$).
\end{proof}

\subsection{$\sing$ preserves fibrations}
\label{SS:preserves}

We are finally ready to prove one of our main results.

\begin{theorem}
\label{T:sing of WSR fib}
 Let $p : \stack{X}\to \stack{Y}$  be a morphism of Serre topological stacks
 that is a (weak) (trivial) Serre fibration and also a Reedy fibration. Then,
 $\sing(p):\sing(\stack{X}) \to \sing(\stack{Y})$ is a (weak) (trivial) Kan fibration.
\end{theorem}

\begin{proof}
   Let  $R_m(\stack{X}):=B(\stack{X})_{*,m}$ be  the
   $m^{\text{th}}$ row of the bisimplicial set 
   $B(\stack{X})$ (where $B(\stack{X})$ is defined in \cref{D:sing}). 
   Note that we have
   \[R_0(\stack{X})=\ob(\stack{X}\Del),
   \ R_1(\stack{X})=\mor(\stack{X}\Del), \
   R_m(\stack{X})=
   R_1(\stack{X})\times_{R_0(\stack{X})}\times\cdots\times_{R_0(\stack{X})}R_1(\stack{X}).
   \]
It follows from \cref{L:ob and mor kan} that, for every $m$, 
$B(p)_{*,m} : B(\stack{X})_{*,m} \to B(\stack{Y})_{*,m}$ is a (weak) (trivial) Kan fibration. 
Furthermore, $B(p)$ is a Reedy fibration of bisimplicial sets because, 
by assumption, $p\Del : \stack{X}\Del\to \stack{Y}\Del$ is  a Reedy fibration 
of simplicial groupoids, and the nerve functor $N : \gpd \to \sset$ preserves 
fibrations and limits (see the proof of Proposition \ref{P:reedy=inj}).  
It follows now from \cref{L:jardine_lemma} that $B(p)$ is a diagonal (weak)  
(trivial) Kan fibration. In other words, $\sing(p):\sing(\stack{X}) \to \sing(\stack{Y})$ 
is a (weak) (trivial) Kan fibration.
\end{proof}

\begin{cor}
 Let $\stack{X}$  be a Reedy fibrant Serre topological stack. Then,
 $\sing(\stack{X})$ is a  Kan simplicial set.
\end{cor}

\begin{cor}
\label{C:WSerre replace}
	For every (weak) (trivial) Serre fibration of Serre stacks $p:\stack{X} \to \stack{Y}$,
    there exists a strictly commutative diagram
	\[
	\begin{tikzpicture}[scale=1]

	\node (a) at (0,1.5) {$\stack{X}$};
	\node (b) at (0,0) {$\stack{X}'$};
	\node (c) at (2,1.5) {$\stack{Y}$};
	
	\draw[->] (a) to node[left] {$\sim$} node[right] {$g$} (b);
	\draw[->] (a) to node[above] {$p$} (c);
	\draw[->] (b) to node[below right] {$p'$} (c);

\end{tikzpicture}
\]
  where $p'$ is a (weak) (trivial) Serre fibration as well as an injective (hence, also Reedy)
  fibration, and $g: \stack{X} \risom\stack{X}'$ is an equivalence of Serre stacks.
  In particular, $\sing(p'): \sing(\stack{X}') \to \sing(\stack{Y})$
  is a (weak) (trivial) Kan fibration.
\end{cor}

\begin{proof}
	This follows from \cref{P:reedy_replace} and \cref{T:sing of WSR fib}. (Also
	see \cref{R:serre}.)
\end{proof}

\begin{cor}
\label{C:SX_Kan}
	For every Serre stack $\stack{X}$ there exists a Serre stack
	$\stack{X}'\sim \stack{X}$ equivalent to it that is Reedy fibrant
	 (hence,  $\sing(\stack{X}')$
	is a Kan simplicial set).
\end{cor}

\section{Singular functor preserves weak equivalences}
\label{S:WE}

In this section, we prove that the singular functor has the correct homotopy type by 
showing that it takes a weak equivalence of topological stacks to a weak equivalence  
of simplicial sets (\cref{T:preserves we}). We begin with a special case.

%
%
%
%

\begin{prop}
\label{P:classifying_we}
	Let $\stack{X}$ be a Serre stack, and let $\varphi: X \to \stack{X}$ be a trivial
	weak Serre fibration with
	$X$ (equivalent to) a topological space
	(i.e., $X$ is a  classifying space for $\stack{X}$ in the sense of \cref{T:classifying}).
	Then,  $\sing(\varphi) : \sing(X) \to \sing(\stack{X})$ is a weak equivalence of
	simplicial sets.
\end{prop}

\begin{proof}
We may assume that $\stack{X}$ is Reedy fibrant
(\cref{C:equiv=>we} and \cref{C:SX_Kan}). 
By \cref{C:equiv=>we} and \cref{C:WSerre replace}, we may assume that 
$\varphi : X \to \stack{X}$  is a trivial
weak Serre fibration as well as a Reedy fibration. 
Note that we are not insisting on
$X$ being  \emph{isomorphic} to  but only equivalent to a
topological space $X'$.

Observe that we can always find a pair of inverse equivalences
between $X$ and $X'$. On the one hand, we have that $\pi_0(X(T))=X'(T)$ for every 
$T \in \topspace$, so we have an
equivalence $p : X \to X'$. In particular, $X(X') \to X'(X')$ is an equivalence of groupoids
(the latter is actually a set).
Picking $f \in X(X')$ in the inverse image of $\id \in X'(X')$ and applying Yoneda's lemma,
we find the desired inverse $f :  X' \to X$ to $p$.

Now, by \cref{T:sing of WSR fib}, $\sing(\varphi) : \sing(X) \to \sing(\stack{X})$ 
is a trivial weak Kan fibration,
and $\sing(\stack{X})$ is Kan. Furthermore, the conditions of \cref{L:trivialgamma} 
are satisfied as
the map $\sing(X') \to \sing(X)$ is a weak equivalence (\cref{C:equiv=>we})
and $\sing(X')$ is Kan.
So, by \cref{L:trivialgamma}, $\sing(\varphi)$ is a weak equivalence.
\end{proof}


\begin{theorem}
\label{T:preserves we}
 Let $f : \stack{X}\to\stack{Y}$ be a weak equivalence of Serre stacks. Then,
 $\sing(f) : \sing(\stack{X})\to\sing(\stack{Y})$ is a weak equivalence of
 simplicial sets.
\end{theorem}
\begin{proof}
 We can choose classifying atlases
 $\varphi : X \to \stack{X}$ and $\psi : Y \to\stack{Y}$
 (in the sense of \cref{T:classifying}) fitting in a
 2-commutative diagram
\[
\begin{tikzpicture}[scale=1]
	\node (a) at (0,2) {$X$};
	\node (b) at (0,0) {$\stack{X}$};
	\node (c) at (2,2) {$Y$};
	\node (d) at (2,0) {$\stack{Y}$};
	
	\draw[->] (a) to node[left] {$\varphi$} (b);
	\draw[->] (a) to node[above] {$f'$} (c);
	\draw[->] (b) to node[below] {$f$} (d);
	\draw[->] (c) to node[right] {$\psi$} (d);
	
	\draw[dbl] (1,2) to [bend left=30] node[below right] {} (0,1);
	
\end{tikzpicture}
\]
This is done as follows. Choose classifying atlases
$\psi: Y \to \stack{Y}$ and  $h : X \to \stack{X}\ttimes_{\stack{Y}}Y$. 
Set $\varphi=\pr_1\circ h$ and $f'=\pr_2\circ h$; by  (\cite{No14}, Lemma 3.8), 
$\varphi$ is again a trivial weak Serre fibration. 

Now, by the two-out-of-three property, $f'$ is a weak equivalence. Applying $\sing$,
we find a homotopy commutative diagram in simplicial sets where
$\sing(f')$, $\sing(\varphi)$ and $\sing(\psi)$ are weak equivalences of
simplicial sets (by \cref{P:classifying_we}, also see \cref{R:sing restricts}).
Therefore, $\sing(f)$ is also a weak equivalence
by the two-out-of-three property.
\end{proof}

\begin{cor}
\label{C:we_classifying_altas}
    Let $\stack{X}$ be  a Serre topological stack, and let 
    $\mathbb{X}=[R \rightrightarrows X]$ be a groupoid presentation for it. Then, 
    there is a natural weak equivalence
	\[
		\sing (\| N(\mathbb{X})\|) \to \sing (\stack{X} ),
	\]
	of simplicial sets, 
	where the left-hand occurrence of $\sing$ is the classical singular chains functor, 
	and $\|-\|$ denotes the fat geometric realization.
\end{cor}

\begin{proof}
	This follows from the fact that there is a natural map 
	$\| N(\mathbb{X})\| \to \stack{X}$, and this map is a classifying space for
	$\stack{X}$; see \cite{No14}, Corollary 3.17 and \cite{No12}, Theorem 6.3.
\end{proof}


\providecommand{\bysame}{\leavevmode\hbox
to3em{\hrulefill}\thinspace}
\providecommand{\MR}{\relax\ifhmode\unskip\space\fi MR }
\providecommand{\MRhref}[2]{%
  \href{http://www.ams.org/mathscinet-getitem?mr=#1}{#2}
} \providecommand{\href}[2]{#2}

\end{document}